\documentclass[10pt,oneside,reqno]{amsart}
\usepackage{amsmath}
\usepackage{paralist}
\usepackage{graphicx} 
\usepackage{epstopdf}
\usepackage[colorlinks=true]{hyperref}
\usepackage{float}
\usepackage{color,comment}
\usepackage{caption}
\usepackage{wrapfig}
\usepackage{multirow}
\usepackage{tabularx}
\usepackage{booktabs}
\usepackage{rotating} 
\usepackage{orcidlink}
\usepackage{subfigure}
\usepackage[toc]{appendix}
\newtheorem{theorem}{Theorem}[section]

\theoremstyle{definition}

\newtheorem{remark}{Remark}[section]


\begin{document}
\title[Prey Aggregation induced Allee Effect]{Dynamics of Diseased-Impacted Prey Populations: Defense and Allee Effect Mechanisms}

\author[Antwi-Fordjour, Overton, Lee]{}
\maketitle

\centerline{\scshape  Kwadwo Antwi-Fordjour$^{1*} {\orcidlink{0000-0002-4961-0462}}$,  Zachary Overton$^{1}$, Dylan Lee$^1$}

\vspace{.7cm}
{\centerline{1) Department of Mathematics and Computer Science,}
 \centerline{ Samford University,}
 \centerline{ Birmingham, AL 35229, USA}
}


\vspace{.7cm}
\centerline{ *Corresponding author's email: kantwifo@samford.edu;}
\smallskip
\centerline{ Contributing authors: zoverton@samford.edu; dlee9@samford.edu;}

\begin{abstract}
This study introduces an innovative framework for merging ecological and epidemiological modeling via the formulation of a sophisticated predator-prey model that addresses the intricacies of disease dynamics, the Allee effect, and defensive mechanisms through prey aggregation. Employing rigorous stability and bifurcation analyses, we identify multiple feasible equilibria and establish critical thresholds that influence population survival and extinction. Our mathematical model reveals that the intensity of the Allee effect plays a crucial role in shaping population recovery and disease persistence, offering pivotal insights into finite time extinction mechanisms. We further illustrate, through extensive numerical simulations, that adjusting susceptible prey aggregation strategically can substantially reduce disease transmission, emphasizing the applicability of our findings for practical conservation interventions. The combined modulation of the aggregation constant and Allee effect determined three primary ecological outcomes: stable coexistence, elimination of infected prey, and complete population extinction. Moreover, these results have significant implications for wildlife management and ecosystem resilience, providing a solid theoretical framework for interdisciplinary strategies aimed at protecting endangered species. 
\end{abstract}

\section{Introduction}
\noindent The interdisciplinary field of eco-epidemiology bridges ecological science and epidemiology to study the complex interplay between infectious diseases and the ecological communities encompassing humans and wildlife \cite{gomez2024eco}. Ecological factors, including species diversity, habitat fragmentation, and climate change, can significantly influence pathogen transmission and host susceptibility  \cite{lunn2021spatial,panigoro2019dynamics}. For example, the biodiversity hypothesis posits that ecosystems with higher species diversity tend to have lower disease prevalence, as diverse communities can dilute the impact of pathogens \cite{keesing2010impacts}. Sahoo and Poria developed a predator-prey model incorporating parasitic infection in the prey \cite{sahoo2014effects}. Their research demonstrated that providing additional food to predators can control parasitic infection in prey species. Alam conducted both numerical and analytical investigations into the risk factors associated with a predator-prey model, incorporating a discrete time delay to account for the gestation period of the prey in predator dynamics. \cite{alam2009risk}. Kumar et al. recently examined the disease transmission dynamics in leptospirosis, a zoonotic disease, by incorporating two discrete delays, including the incubation period delay \cite{kumar2025dynamical}. Ultimately, eco-epidemiology provides a vital framework for addressing the challenges posed by infectious diseases.

The Allee effect is a fundamental ecological mechanism that describes how species interactions influence population dynamics \cite{sahoo2024crucial}. It is characterized by a positive correlation between population size and individual fitness, meaning that smaller populations experience reduced per capita growth rates, which can lead to further declines and an increased risk of extinction \cite{allee1927animal,courchamp2008allee}. This mechanism is particularly significant in the conservation of endangered species, as it can create a vicious cycle where declining numbers make recovery increasingly difficult. Understanding the role of the Allee effect is therefore essential for effective conservation strategies and ecosystem management, as it helps predict population persistence and informs efforts to prevent extinction \cite{kumbhakar2023bistability,dennis1989allee,saifuddin2017eco}. A weak Allee effect manifests as a reduction in per capita growth rate at low population densities, yet crucially, it lacks a critical population threshold below which decline is inevitable. Conversely, a strong Allee effect is distinguished by the presence of such a threshold; below this critical population size, the per capita growth rate becomes negative, predisposing the population to extinction \cite{fadai2020unpacking}. Furthermore, eco-epidemiological models investigate how changes in prey and predator populations, influenced by various ecological dynamics such as the Allee effect, can result in different patterns of disease emergence and persistence \cite{ma2023asymptotic}.  Shaikh and Das examined how the Allee effect influences an eco-epidemic predator-prey model with a Holling Type II functional response \cite{shaikh2020eco}. Sharma and Samanta proposed a ratio-dependent predator-prey model that incorporates disease in the prey population with Allee effects \cite{sharma2015ratio}. Also, see \cite{kumar2020impact} for similar research on ratio-dependent prey-predator model model but with disease in predator and subjected to strong Allee effect in prey. Recently, Samanta et al. observed multistability and chaos in their proposed eco-epidemiological system with Allee effect. Again, the authors observed that, a strong Allee effect can prevent the system from going into chaos, but once a certain threshold is achieved, the prey population is eradicated leading to the eventual eradication of the predator population. \cite{samanta2024multistability}.

The predator-prey model with prey defense mechanisms via Rosenzweig functional response (or prey aggregation) have been studied by several researchers, see \cite{R71,antwi2024dual,B12,V18,M18} and references therein. Prey aggregation is a crucial ecological phenomenon that influences population dynamics and species interactions within ecosystems \cite{godfray1992aggregation}. This behavior, characterized by individuals of a species coming together in groups, can serve multiple functions, such as enhancing foraging efficiency, providing protection from predators, and facilitating social interactions \cite{sumpter2006principles,krausz2013living}. Aggregation is particularly important in the context of disease dynamics, as it can affect transmission rates among populations. Farrell et al. investigated power incidences functions that depend on the number of susceptibles and infectives by powers strictly between 0 and 1 \cite{farrell2018fatal}. These powers are well known to be associated with host extinction \cite{APB20,KV16}. Despite the recognized importance of both susceptible prey aggregation and the Allee effect in shaping population dynamics, no studies have yet examined their combined influence on disease prevalence within prey populations. This lack of investigation represents a significant knowledge gap, and addressing it is vital for advancing our understanding of eco-epidemiological interactions.

The current study is designed to provide a rigorous examination of the interconnected roles of the Allee effect, disease transmission, and host defense mechanisms in shaping eco-epidemiological outcomes. By addressing the following specific research questions, this work aims to contribute novel insights to the field:
\begin{itemize}
\item[(i)] Investigate how the strength of the Allee effect influences the dynamics of the prey population in the presence of infectious disease.
\item[(ii)] Identify targeted strategies that utilize susceptible prey aggregation to mitigate disease transmission.
\item[(iii)] Analyze the ecological significance of multiple equilibrium states within the predator-prey-disease framework.
\item[(iv)] Explore the bifurcation structures associated with the proposed model, highlighting critical transitions in system behavior.
\item[(v)] Assess the combined impact of susceptible prey aggregation and the Allee effect on disease persistence and population stability.
\end{itemize}

This paper is organized as follows: Section \ref{sec:model formulation} presents the formulation of our mathematical model, detailing the assumptions and governing equations. Section \ref{sec:preliminary results} establishes the fundamental properties of the model, including nonnegativity and boundedness of solutions. In Section \ref{sec:equilibria and stability}, we analyze the equilibrium points and their local stability. Section \ref{sec:bifucation analysis} provides a bifurcation analysis, illustrating key dynamical transitions in the model. Section \ref{sec: Disease managemment} explores disease management strategies by targeting susceptible prey aggregation, while Section \ref{sec: Allee-driven extinction} investigates Allee effect-driven finite time extinction. Finally, Section \ref{sec:conclusion} concludes with a discussion on broader implications and future research directions.

\section{The Mathematical Model}\label{sec:model formulation}
\noindent We develop a mathematical model to explore the dynamics of a predator-prey system where the prey population is subject to an infectious disease. Consequently, the prey population is divided into susceptible ($S(t)$) and infected ($I(t)$) classes at time $t$. The predator population is represented by $P(t)$. Our model is constructed based the following biological meaningful assumptions:

\begin{itemize}
\item[(i)] Susceptible prey contract the infection via direct contact with infected individuals, reflecting typical ecological transmission modes. Predators are assumed to be immune to the disease.
\item[(ii)] Infection in prey leads to a decline in reproductive capacity and a diminished ability to compete for resources, effectively reducing their ecological fitness. 
\item[(iii)]  Infected prey are assumed to experience permanent infection, consistent with diseases lacking significant recovery.
\item[(iv)] We assume that the ecosystem's carrying capacity is determined by the total prey population, encompassing both susceptible and infected population. This reflects the ecological interaction where both prey classes utilize shared resources. The continuous growth function is given by:
\begin{align*}
\frac{dS}{dt} &= a_0 S \left(1-\frac{S+I}{K}\right),
\end{align*}
where $a_0$ denotes prey birth rate and $K$ is the environmental carrying capacity of the susceptible and infected prey population. 
\item[(v)]  We assume that the susceptible prey population aggregates when they encounter predators to protect themselves \cite{antwi2024dual}. This is modeled by $S^r$, where $r\in (0,1)$ is the aggregating constant. 
\item[(vi)] When susceptible prey encounter predators, their defensive behaviors create an Allee effect within their population. The continuous growth function for the susceptible prey population now becomes:
\begin{align*}
\frac{dS}{dt} &= a_0 S \left(1-\frac{S+I}{K}\right)(S-L).
\end{align*}
To quantify the strength of the Allee effect, we define a threshold $L$. A weak Allee effect is present in the regime $-K<L\leq  0$, indicating a limited impact on population growth at low densities, see Figure \ref{fig:perCapitaGrowth}. Conversely, a strong Allee effect, in the regime $0<L<K$, signifies a more pronounced negative impact on growth at low densities \cite{debnath2020global}. The per-capita growth rate of the susceptible prey population is given by the rate of change of the susceptible prey population divided by the susceptible  population size. The per-capita growth rate of the susceptible prey is given by
\begin{align*}
\frac{1}{S}\frac{dS}{dt}= a_0 \left(1-\frac{S+I}{K}\right)(S-L).
\end{align*}

\begin{figure}[hbt!]
\begin{center}
    \includegraphics[width=8cm, height=7cm]{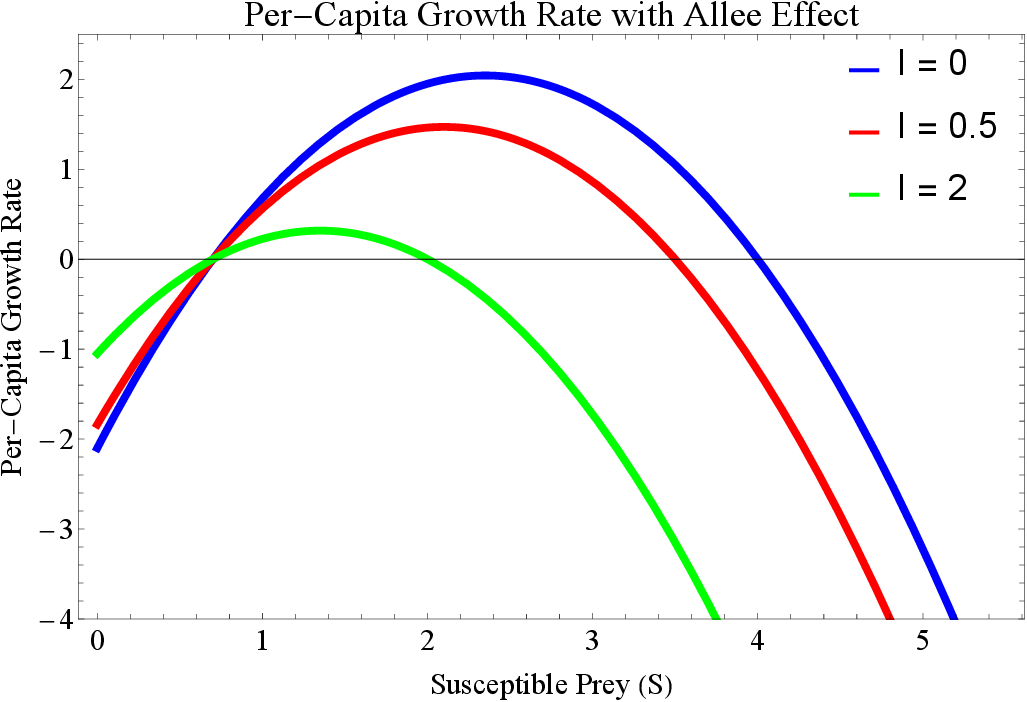}
\end{center}
\caption{Per-capita growth rate of the susceptible prey population for different fixed infected prey values, i.e., $I=0$, $I=0.5$ and $I=2$. The parameters are fixed as $a_0=3,~K=4,~L=0.7$}
\label{fig:perCapitaGrowth}
\end{figure}
\item[(vii)] The predator and infected prey interactions are modeled via mass action (or type I response function), with consumption rate increasing linearly to a saturation point. Mass action also governs susceptible and infected prey interactions.
\item[(viii)] The model incorporates natural death rates for each population: susceptible prey, infected prey, and predator.

\end{itemize}

\noindent Building upon the stated assumptions, we formulate the interaction between susceptible prey, infected prey, and predators (SIP) through the following model:
\begin{align}
\nonumber \frac{dS}{dt} &= a_0 S \left(1-\frac{S+I}{K}\right)(S-L)  - d_0S^rP - e_0SI = W_1(S,I,P), \\
   \frac{dI}{dt} &= -a_1I + e_0SI - d_1IP = W_2(S,I,P), \label{Mainsystem}  \\
\nonumber     \frac{dP}{dt} &= -a_2P + d_2S^rP + d_3IP = W_3(S,I,P),
\end{align}
with initial conditions
\begin{align*}
S(0)\geq 0,~I(0)\geq 0,~P(0)\geq 0.
\end{align*}

\begin{table}[htbp]
    \caption{Biological description of variables/parameters used in model \eqref{Mainsystem}.}
    \label{tab:description}    
    \begin{tabular}{| c | p{0.8\textwidth} |}
        \hline 
        Variable/Parameter & Biological Description \\
 \hline
 $S$ & Susceptible prey density \\ 
 $I$ & Infected prey density  \\ 
 $P$ & Predator density  \\
 $a_0$ & Susceptible prey natural birth rate \\ 
 $d_0$ & Attack rate of the predator on the susceptible prey\\ 
 $e_0$ & Disease transmission rate\\
 $K$ & Carrying capacity \\
 $L$ & Allee threshold \\ 
 $r$ & Aggregating constant\\
 $a_1$ & Infectious prey natural death rate \\
 $d_1$ & Attack rate of the predator on infectious prey\\
 $a_2$ & Predator natural death rate \\
 $d_2$ & Constant representing the rate at which predators convert susceptible prey biomass into their own, accounting for both attack rate and conversion efficiency \\
 $d_3$ & Constant representing the rate at which predators convert infected prey biomass into their own, accounting for both attack rate and conversion efficiency \\
 \hline
\end{tabular}
\end{table}

\noindent The biological description of variables/parameters used in model \eqref{Mainsystem} are listed in Table \ref{tab:description}. All parameters are assumed to be positive except the Allee threshold. We shall assume biological meaningful parameter values and initial conditions for this study.

\section{Preliminary Results}\label{sec:preliminary results}
\noindent To make sure the model \eqref{Mainsystem} is both biologically and mathematically meaningful, we first examine whether its solutions are well-posed. This means showing that the population values remain nonnegative and within realistic limits. Nonnegativity ensures that population sizes never become negative, which would be biologically impossible. Boundedness reflects natural resource limits and carrying capacity, preventing unrealistic population growth. These properties are essential for understanding how the system behaves over time and its ecological significance.
\subsection{Nonnegativity}
\begin{theorem}\label{thm:nonnegativity}
All solutions $(S(t), I(t), P(t))$ of the model \eqref{Mainsystem} are nonnegative for all $t\geq 0$.
\end{theorem}

\begin{proof}
The right hand side of model \eqref{Mainsystem} is continuous and locally non-smooth function of the dependent variable $t$. We obtain the following after integration.
\begin{align*}
S(t)&=S(0)\text{exp}\left(\displaystyle\int_0^t \Big[a_0\left(1-\frac{S+I}{K}\right)(S-L) - d_0S^{r-1}P - e_0 I\Big] ds\right)\geq 0 \\
I(t)&=I(0)\text{exp}\left(\displaystyle\int_0^t \Big[-a_1 + e_0S - d_1P\Big]   ds\right)\geq 0\\
P(t)&=P(0)\text{exp}\left(\displaystyle\int_0^t \Big[-a_2 + d_2S^r + d_3I\Big] ds\right)\geq 0
\end{align*}

\noindent Therefore, all solutions initiating from the interior of 

\begin{equation*}
    \mathbb{R}^3_+ = \{(S(t), I(t), P(t)): S(t)\geq 0, I(t)\geq 0,  P(t)\geq 0\}.
\end{equation*}

remain in it for all future time.
\end{proof}

\subsection{Boundedness}

\begin{theorem}
All the solutions $(S(t), I(t), P(t))$ of the model \eqref{Mainsystem} with positive initial conditions are uniformly bounded if $d_0 > d_2$ and $d_1 > d_3$.
\end{theorem}

\begin{proof}
Let us define the function $R\left(S(t),I(t),P(t)\right) = S(t) + I(t) + P(t)$. Then
\begin{align*}
    \dfrac{dR}{dt} &= \dfrac{dS}{dt} + \dfrac{dI}{dt} + \dfrac{dP}{dt}\\
    &= a_0 S \left(1-\frac{S+I}{K}\right)(S-L)  - d_0S^rP - e_0SI-a_1I + e_0SI - d_1IP -a_2P \\
    & + d_2S^rP + d_3IP  \\
    &= a_0 S \left(1-\frac{S+I}{K}\right)(S-L) - a_1I -a_2P  + (d_2-d_0)S^rP + (d_3- d_1)IP \\
\end{align*}
Let $\mu \in \mathbb{R^+}$, where $\mu \le \min\{ a_1,a_2\}$ and let $L+K>0$ be the upper bound of $S$. Assume $d_0\ge d_2$ and $d_1\ge d_3$. We obtain the following differential inequality:
\begin{align*}
   \frac{dR}{dt}+\mu R &\leq  a_0 S \left(1-\frac{S+I}{K}\right)(S-L) +\mu S + (\mu -a_1)I +(\mu -a_2)P \\
 &\leq  a_0 S \left(1-\frac{S+I}{K}\right)(S-L) +\mu S  \\
   & =  S \left[a_0 \left(1-\frac{S+I}{K}\right)(S-L) +\mu \right] \\
    & \le  S \left(-\dfrac{a_0 S^2}{K} + \dfrac{a_0 S (L+K)}{K} +\mu \right) \\
      & \le  (L+K) \left(-\dfrac{a_0 S^2}{K} + \dfrac{a_0 S (L+K)}{K} +\mu \right) \\
       & \le  (L+K) \left(\mu  +\dfrac{a_0 (L+K)^2}{4K}\right) \\
\end{align*}

\noindent Take $Q=(L+K) \left(\mu  +\dfrac{a_0 (L+K)^2}{4K}\right)$. Thus by applying the theorem on differential inequality, we obtain
\[0\leq R(S(t),I(t),P(t))\leq \frac{Q(1-e^{-\mu t})}{\mu} + R(S(0),I(0),P(0))e^{-\mu t},\]
which implies
\[\limsup_{t\rightarrow \infty}R(S(t),I(t),P(t))\leq \frac{Q}{\mu}.\]

\noindent Hence, all the solutions of model \eqref{Mainsystem} which initiated in $\mathbb{R}_+^3$ are confined in 
\begin{equation*}
    \Upsilon = \{(S, I, P) \in \mathbb{R}^3_+ : R(S(t), I(t), P(t)) \leq \frac{Q}{\mu} + \epsilon, \epsilon \in \mathbb{R}\}.
\end{equation*}
\end{proof}

\section{Equilibria and their Local Stability}\label{sec:equilibria and stability}
\noindent In this Section, we investigate the existence of ecologically meaningful equilibria for the model \eqref{Mainsystem} and then study their stability. 
\subsection{Equilibria}
The analysis of the ecologically feasible equilibria of model \eqref{Mainsystem} involves considering the susceptible prey, infected prey, and predator nullclines, expressed by
\begin{align}
   \label{susc=0} 0 &= S\left( a_0  \left(1-\frac{S+I}{K}\right)(S-L) - d_0S^{r-1}P - e_0I\right) \\
    \label{inf=0} 0 &= I(-a_1 + e_0S - d_1P) \\
   \label{pred=0} 0 &= P(-a_2 + d_2S^r + d_3I).
\end{align}
We obtain below the ecologically feasible equilibria.
\begin{itemize}
\item[(i)] Extinction equilibrium $E_0=(0,0,0)$.
\item[(ii)] There are two susceptible prey only equilibrium points, namely\\
\begin{enumerate}
\item[(a)] $E_1^1=(S_1^1,0,0)$, where $S_1^1=K$. 
\item[(b)] $E_1^2=(S_1^2,0,0)$, where $S_1^2=L$. 
\end{enumerate}
\item[(iii)] Predator free $E_2=(S_2,I_2,0)$, where $S_2=\frac{a_1}{e_0}$ and $I_2=-\frac{a_0 \left(a_1-e_0 K\right) \left(a_1-e_0 L\right)}{e_0 \left(-a_0 e_0 L+a_0 a_1+e_0^2 K\right)}$. $E_2$ is meaningful provided $I_2>0$.
\item[(iv)] Infectious prey free $E_3=(S_3,0,P_3)$, where $S_3=\left(\frac{a_2}{d_2} \right)^{1/r}$ and\\
 $P_3=-\frac{a_0 \left(\left(\frac{a_2}{d_2}\right){}^{1/r}\right){}^{1-r} \left(\left(\frac{a_2}{d_2}\right){}^{1/r}-K\right) \left(\left(\frac{a_2}{d_2}\right){}^{1/r}-L\right)}{d_0 K}$. $E_3$ is meaningful provided $P_3>0$.
\item[(v)] Coexistence $E_4=(S_4,I_4,P_4)$.
Finding positive solutions for the model \eqref{Mainsystem} analytically is challenging. Therefore, we investigate sufficient conditions that ensure the biological relevance of the coexistence equilibrium point. The coexistence (or endemic) equilibrium point(s) can be determined using the nullcline equations provided below.
\begin{align}
   \label{susc=S*} 0 &= a_0 \left(1-\frac{S+I}{K}\right)(S-L)  - d_0S^{r-1}P - e_0I \\
    \label{inf=I*} 0 &= -a_1 + e_0S - d_1P  \\
   \label{pred=P*} 0 &= -a_2 + d_2S^r + d_3I.
\end{align}   
provided $S>0,I>0,P>0$
\end{itemize}
Solving equation \eqref{pred=P*} allows us to express $I^*$ as a function of the susceptible prey ($S^*$).
\begin{equation}\label{coexistence S*}
I^*=\dfrac{a_2-d_2S^{*r}}{d_3}.
\end{equation}

\noindent Furthermore, we note that $S^*<\left(\frac{a_2}{d_2}\right)^{1/r}=S_3$ when $I^*>0$. Now, we solve for $P^*$ equation \eqref{inf=I*},
\begin{equation}\label{eqn:quad-4-interior}
P^*=\dfrac{-a_1+e_0 S^*}{d_1}.
\end{equation} 
We observe that $P^*>0$ when $S^*>\frac{a_1}{e_0}=S_2$. Combining $S_2$ and $S_3$, we obtain sufficient condition for the existence of $S^*$,
\begin{equation}\label{ineq:condition in S*}
0<S_2<S^* < S_3.
\end{equation}

\begin{remark}\label{remark:equilibrium points}
From Figure \ref{fig:equilibriumPoints}, we observe that the number of interior equilibria varies with the Allee threshold. In Figure \ref{fig:equilibriumPoints}(a), a single interior equilibrium exists at $E_4 = (2.6134, 0.7875, 2.7887)$ when $L=-0.5$. In contrast, Figure \ref{fig:equilibriumPoints}(b) shows the presence of two interior equilibria, located at $E_4^1 = (2.4296, 0.8310, 2.5524)$ and $E_4^2 = (0.9578, 1.2660, 0.6600)$ when $L=-0.1$. 
\end{remark}

\begin{figure}[hbt!]
\begin{center}
\subfigure[]{
    \includegraphics[width=6.15cm, height=5cm]{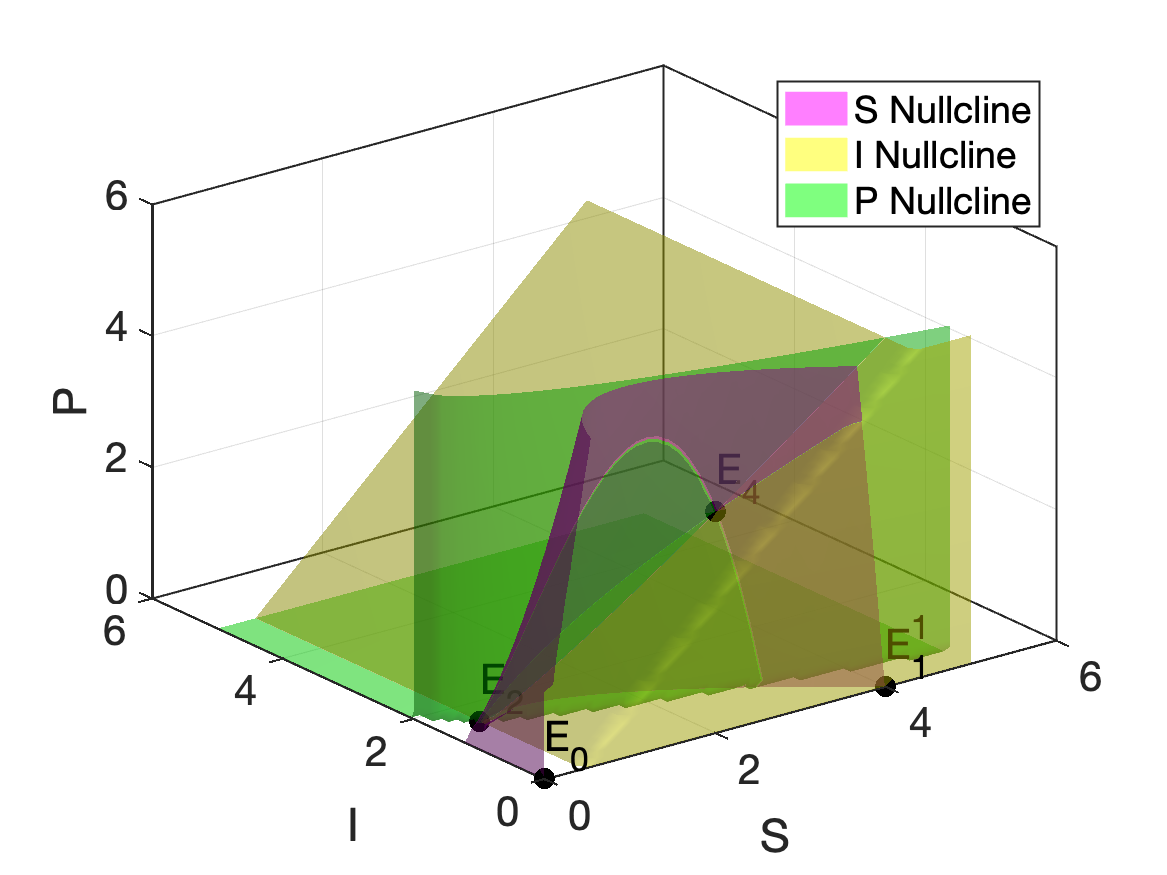}}
\subfigure[]{    
    \includegraphics[width=6.15cm, height=5cm]{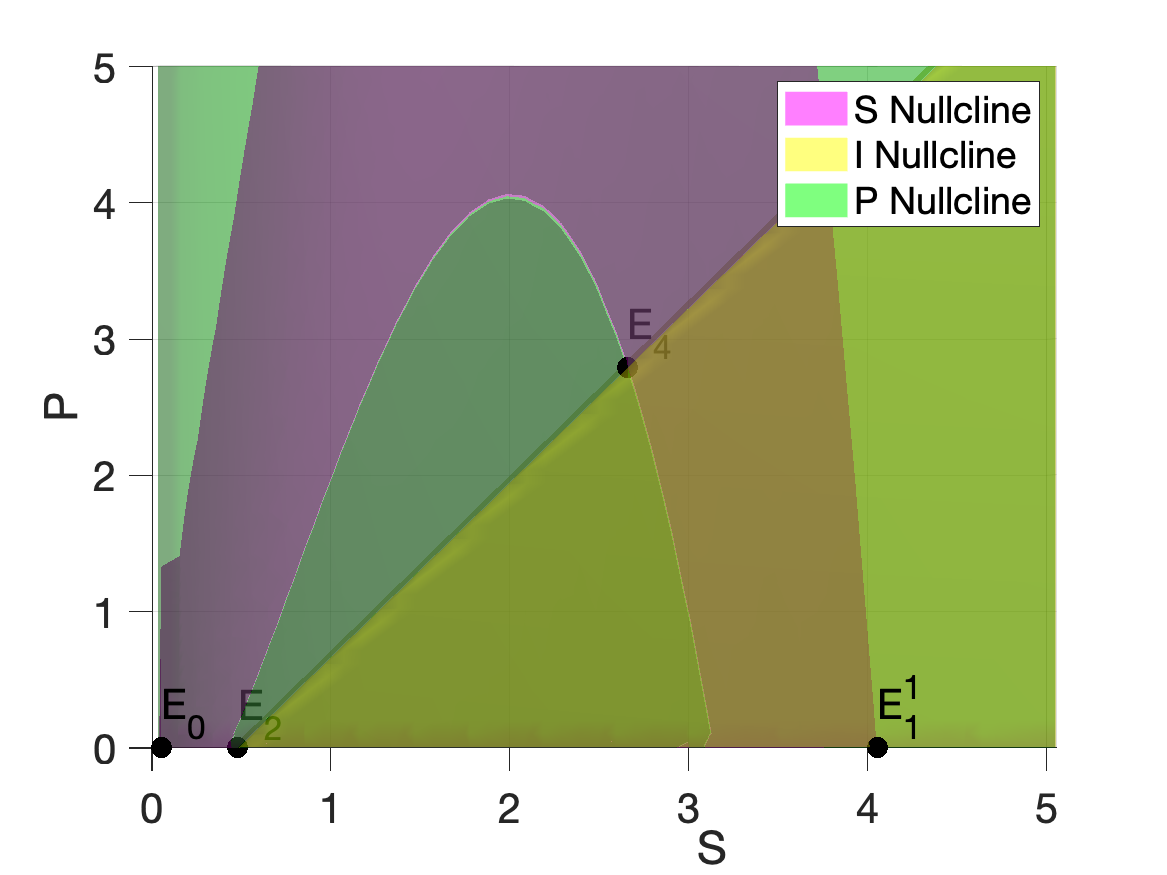}}
\subfigure[]{
    \includegraphics[width=6.15cm, height=5cm]{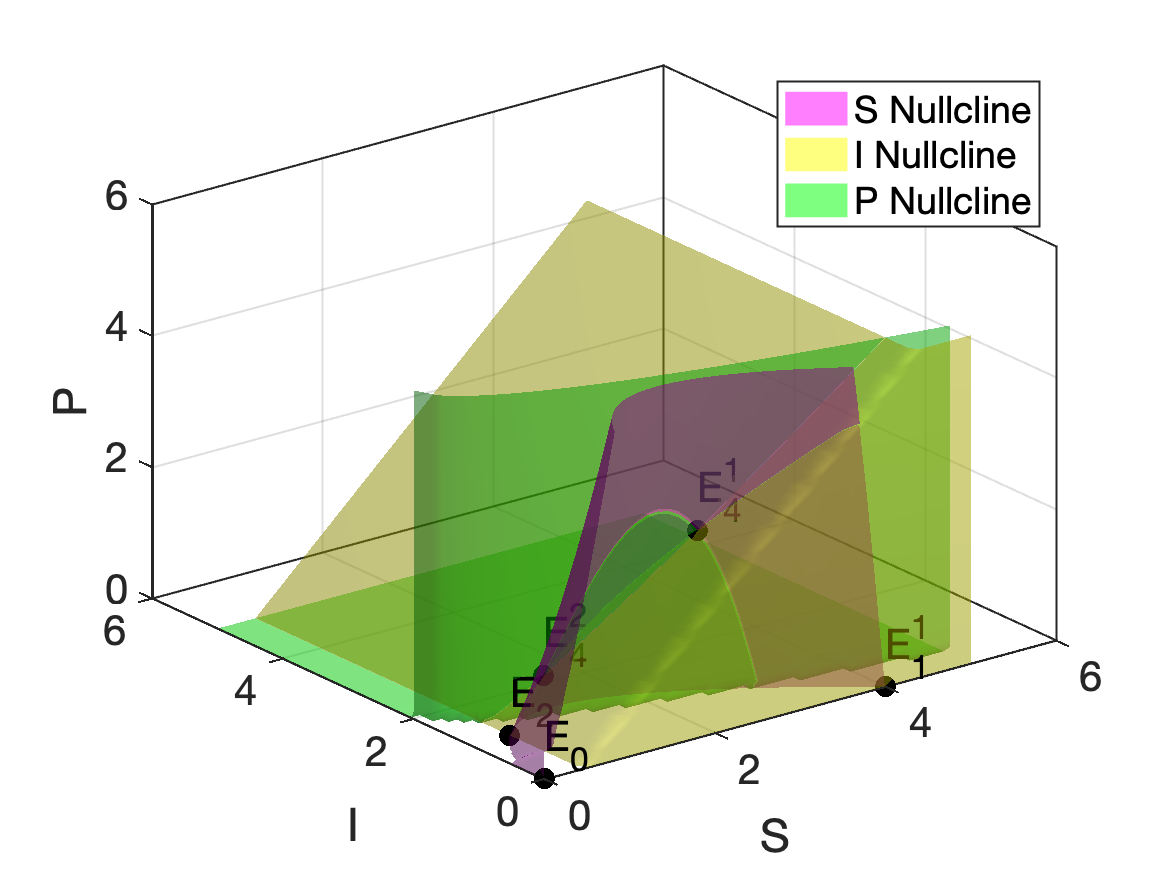}}
\subfigure[]{    
    \includegraphics[width=6.15cm, height=5cm]{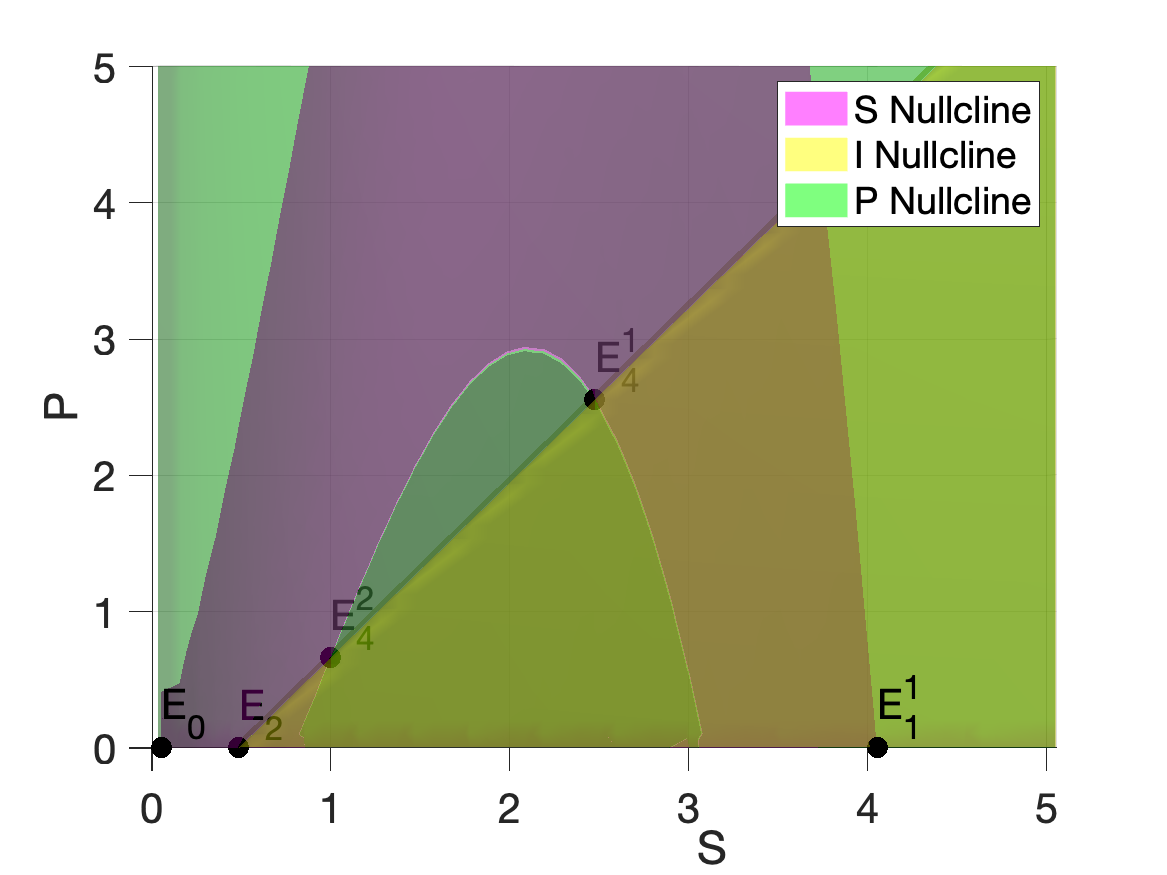}}          
\end{center}
\caption{Intersection of nullclines for varying values of $L$. (a) 3D plot of the nullclines when $L=-0.5$, showing one interior equilibrium point. (b) 2D projection of the S-P nullclines corresponding to (a). (c) 3D plot of the nullclines when $L=-0.1$. (d) 2D projection of the S-P nullclines corresponding to (c). All other parameters are fixed and given as $a_0=3,~a_1=0.4,~a_2=0.8,~r=0.5,~d_0=0.4,~d_1=0.7,~d_2=0.3,~d_3=0.4,~e_0=0.9,~K=4$}
\label{fig:equilibriumPoints}
\end{figure}

\subsection{Local stability analysis}
\noindent The Jacobian matrix for model \eqref{Mainsystem} at any equilibrium point can be expressed in the following form, 
\[
\mathbb{J}=
\begin{bmatrix}
 J_{11}& 
     J_{12} & 
    J_{13}\\[1ex] 
  \ J_{21} & 
     J_{22} & 
     J_{23} \\[1ex]
  J_{31} & 
    J_{32} & 
     J_{33}
\end{bmatrix}
\]
where
\begin{align*}
   J_{11} &=\frac{a_0 \left(-K L+2 K S+2 L S+L I-3 S^2-2 S I\right)}{K}-\frac{d_0 \left(rPS^r+S\right)}{S}-e_0 I\\
     J_{12}&= -\frac{S \left(a_0 (S-L)+e_0 K\right)}{K}\\
     J_{13} &= -d_0 S^r<0\\  
   J_{21} &= e_0 I>0\\
    J_{22} &= -a_1 -d_1P+e_0 S\\
    J_{23} &=-d_1 I<0\\
   J_{31} &= rd_2S^{r-1}P>0\\
    J_{32}&= d_3P>0\\
   J_{33} &= -a_2 +d_2S^r+d_3I
\end{align*}

\begin{remark}\label{remark: E_0}
The extinction equilibrium point $E_0$ cannot be analyzed using the Jacobian matrix linearization method because some terms become singular at the origin. Therefore, we will not include the analysis of $E_0$ in this study.
\end{remark}

\begin{theorem}
The locally asymptotic stability of the equilibrium point $E^i_1$, is established under the conditions: $\frac{a_0 \left(-K L+2 K S_1^i+2 L S_1^i-3 (S_1^{i})^2\right)}{K}<d_0,~ e_0 S_1^i <a_1$ and $a_2> d_2 (S_1^{i})^r$, where $i=1,2$. 
\end{theorem}

\begin{proof}
The Jacobian matrix evaluated at $E^i_1$ is given by
\[\mathbb{J}_{E^i_1}=
\begin{bmatrix}
\frac{a_0 \left(-K L+2 K S_1^i+2 L S_1^i-3 (S_1^{i})^2\right)}{K}-d_0& 
   -\frac{S_1^i \left(a_0 (S_1^i-L)+e_0 K\right)}{K}& 
   d_0 (S_1^{i})^r\\[1ex] 
    0 & 
   e_0 S_1^i -a_1 & 
    0 \\[1ex]
    0 & 
    0 & 
    d_2 (S_1^{i})^r-a_2
\end{bmatrix}\]

\noindent The eigenvalues obtained from the $\mathbb{J}_{E^i_1}$ are $\lambda_1= \frac{a_0 \left(-K L+2 K S_1^i+2 L S_1^i-3 (S_1^{i})^2\right)}{K}-d_0,~\lambda_2= e_0 S_1^i -a_1$ and $\lambda_3= d_2 (S_1^{i})^r-a_2$. The equilibrium point $E^i_1$ will be locally asymptotically stable if all the eigenvalues are negative or have negative real parts. Thus, the model \eqref{Mainsystem} at $E^i_1$ is locally asymptotically stable if $\frac{a_0 \left(-K L+2 K S_1^i+2 L S_1^i-3 (S_1^{i})^2\right)}{K}<d_0,~ e_0 S_1^i <a_1$ and $a_2> d_2 (S_1^{i})^r$. Note that $S^1_1=K$ and $S^2_1=L$.
\end{proof}

\begin{theorem}\label{thm:stable E_2} 
The local asymptotic stability of the predator-free equilibrium $E_2$ is established when the condition  $C_{11}<0,~C_{33}<0$, and $C_{12}C_{21}<0$ are satisfied.
\end{theorem}

\begin{proof}
The Jacobian matrix evaluated at $E_2$ is given by
\begin{equation}\label{JacE2}
\mathbb{J}_{E_2}=
\begin{bmatrix}
C_{11} & 
   C_{12}& 
   C_{13}\\[1ex] 
   C_{21} & 
   0 & 
  C_{23}\\[1ex]
    0 & 
    0 & 
    C_{33}
\end{bmatrix}
\end{equation}
where
\begin{align*}
C_{11}&=\frac{a_0 \left(-K L+2 K S_2+2 L S_2+L I_2-3 S_2^2-2 S_2 I_2\right)}{K}-d_0-e_0 I_2 \\
C_{12}&=-\frac{S_2 \left(a_0 (S_2-L)+e_0 K\right)}{K} \\
C_{13}&= -d_0 S_2^r<0\\
C_{21}&= e_0 I_2>0\\
C_{23}&= -d_1 I_2<0 \\
C_{33}&= -a_2 +d_2S_2^r+d_3I_2
\end{align*}
The eigenvalues of the Jacobian matrix at $E_2$ are $\lambda_1=C_{33}$ and the roots of the characteristic polynomial given by
\begin{equation}\label{charac:s2}
\lambda^2-C_{11}\lambda-C_{12}C_{21}=0.
\end{equation}
For this characteristic polynomial, the roots are $\lambda_2+\lambda_3=C_{11}$ and  $\lambda_2\lambda_3=-C_{12}C_{21}$. Hence, $E_2$ is locally asymptotically stable if $C_{11}<0,~C_{33}<0$, and $C_{12}C_{21}<0$.
\end{proof}

\begin{theorem}\label{thm:Stable E_3}
The local asymptotic stability of the infectious prey-free equilibrium $E_3$ is established when the conditions $H_{22} < 0$ and $H_{11} < 0$ are satisfied. 
\end{theorem}

\begin{proof}
The Jacobian matrix evaluated at $E_3$ is given by
\begin{equation}\label{JacE3}
\mathbb{J}_{E_3}=
\begin{bmatrix}
 H_{11}& 
    H_{12}
     & H_{13}\\[1ex] 
  0 & 
  H_{22} & 
    0\\[1ex]
H_{31} & 
   H_{32} & 
   0
\end{bmatrix}
\end{equation}
where
\begin{align*}
H_{11}&=\frac{a_0 \left(-K L+2 K S_3+2 L S_3-3 S_3^2\right)}{K}-\frac{d_0 \left(rP_3S_3^r+S_3\right)}{S_3}\\
H_{12}&=-\frac{S_3 \left(a_0 (S_3-L)+e_0 K\right)}{K}\\
H_{13}&=-d_0S_3^r <0\\
H_{22}&=  -a_1 -d_1P_3+e_0 S_3\\
H_{31}&=rd_2S_3^{r-1}P_3>0\\
H_{32}&= d_3P_3>0\\
H_{33}&= 0
\end{align*}
The eigenvalues of the Jacobian matrix at $E_3$ are $\lambda_1=H_{22}$ and the roots of the characteristic polynomial given by
 \begin{equation}\label{charac:s3}
 \lambda^2-H_{11}\lambda-H_{13}B_{31}=0.
 \end{equation}
For the characteristic polynomial in \eqref{charac:s3}, the roots are $\lambda_2+\lambda_3=H_{11}$ and  $\lambda_2\lambda_3=-H_{13}H_{31}>0$. Thus, $E_3$ is locally asymptotically stable if $H_{22}<0$ and $H_{11}<0$.
\end{proof}

\begin{theorem}\label{thm:coexistence}
The local asymptotic stability of the coexistence equilibrium $E_4$ is determined by the conditions $\Omega_1 > 0$, $\Omega_3 > 0$, and $\Omega_1\Omega_2 > \Omega_3$. 
\end{theorem}

\begin{proof}
The Jacobian matrix evaluated at $E_4$ is given by
\[
\mathbb{J}_{E_4}=
\begin{bmatrix}
 F_{11}& 
    F_{12}
     & F_{13}\\[1ex] 
  F_{21} & 
   0 & 
    F_{23}\\[1ex]
F_{31} & 
   F_{32} & 
  0
\end{bmatrix}
\]
where
\begin{align*}
 	F_{11} &=\frac{a_0 \left(-K L+2 K S^*+2 L S^*+L I^*-3 S^{*2}-2 S^* I^*\right)}{K}-\frac{d_0 \left(rP^*S^{*r}+S^*\right)}{S^*}-e_0 I^*\\
     F_{12}&= -\frac{S^* \left(a_0 (S^*-L)+e_0 K\right)}{K}\\
     F_{13} &=  -d_0S^{*r}\\  
   F_{21} &= e_0I^*\\
   F_{22} &= 0\\
    F_{23} &=-d_1I^*\\
   F_{31} &= rd_2S^{*r-1}P^*\\
    F_{32}&= d_3P^*\\
   F_{33} &= 0
\end{align*}

\noindent The characteristic equation of the Jacobian matrix $\mathbb{J}_{E_4}$ around the coexistence equilibrium point $E_4$ is given by
\begin{equation}\label{eqn:character of E_4}
\lambda^3+\Omega_1\lambda^2+\Omega_2\lambda+\Omega_3=0
\end{equation}

where 
\begin{align*}
\Omega_1&=-F_{11}\\
\Omega_2&=-\left(F_{23}F_{32}+F_{12}F_{21}+F_{13}F_{31}\right)\\
\Omega_3&=F_{11}F_{23}F_{32}-F_{12}F_{23}F_{31}-F_{13}F_{21}F_{32}  
\end{align*}

\noindent From Routh–Hurwitz criteria, $E_4$ is locally asymptotically stable provided $\Omega_1 > 0$, $\Omega_3 > 0$, and $\Omega_1\Omega_2 > \Omega_3$, since all the eigenvalues of the characteristic equation \eqref{eqn:character of E_4} are negative or will have negative real parts.
\end{proof}

\section{Bifurcation Analysis}\label{sec:bifucation analysis}
\subsection{Co-dimension one bifurcation}
As a system parameter traverses a critical threshold, co-dimension one bifurcations manifest as significant alterations in the system's dynamics. These bifurcations involve the creation, merging, or annihilation of equilibrium points and/or periodic orbits, thus reshaping the system's phase portrait.
%

%
\subsubsection{Transcritical bifurcation}
In a transcritical bifurcation, the stability of two equilibrium points is interchanged.  As a parameter passes through a critical value, the previously stable equilibrium becomes unstable, and the previously unstable equilibrium becomes stable.  Importantly, both equilibria are present on either side of the bifurcation point.

\begin{theorem}\label{thm:transcriticalk1}
The model \eqref{Mainsystem} undergoes transcritical bifurcation around the predator free equilibrium $E_2$ at critical parameter value $L=L^{TC}$, where $L^{TC}$ is computed when $C_{12}|_{L=L^{TC}}C_{21}=0$. The values of $C_{ij}$ for $i,j=1,2$ can be found in equation \eqref{JacE2}.
\end{theorem}

\begin{proof}
Analysis of the Jacobian matrix \eqref{JacE2} of system \eqref{Mainsystem} at the equilibrium point $E_2$ reveals that $C_{33}$ is an eigenvalue.  The remaining eigenvalues are solutions to the quadratic equation \eqref{charac:s2}.  A necessary condition for the existence of a zero eigenvalue is given by:
\[
 C_{12} C_{21} = 0
\]
is satisfied, which corresponds to \( L = L^{TC} \).

We proceed by identifying the eigenvectors \( X_1 \) and \( X_2 \) corresponding to the zero eigenvalue of the matrices \( J_{E_2}^{TC} \) and \( (J_{E_2}^{TC})^T \), respectively. These eigenvectors are given by:
\[
X_1 = \left(-\frac{C_{23}}{C_{11}}\gamma_1, 0, \gamma_1\right)^T, \quad X_2 = \left(-\frac{C_{21}}{C_{11}}\gamma_2, \gamma_2, 0\right)^T,
\]
where \( \gamma_1 \) and \( \gamma_2 \) are arbitrary real numbers.Additionally, the following conditions hold:
\begin{itemize}
    \item \( X_2^{T} W_{L}\left(E_2, L^{TC}\right) = 0 \),
    \item \( X_2^{T} \left[ DW_{L}\left(E_2, L^{TC}\right) X_1 \right] \neq 0 \),
    \item \( X_2^{T} \left[ D^2 W\left(E_2, L^{TC}\right)(X_1, X_1) \right] \neq 0 \).
\end{itemize}

\noindent These conditions confirm that model \eqref{Mainsystem} undergoes a transcritical bifurcation at the equilibrium point \( E_2 \) when \( L = L^{TC} \).

\end{proof}

\subsubsection{Hopf bifurcation}
A Hopf bifurcation is a critical point in the behavior of a dynamical system (often described by differential equations) where a change in a system parameter causes a shift in the system's stability.

\begin{theorem}\label{thm:hopfk2}
The model \eqref{Mainsystem} undergoes a Hopf bifurcation near the coexistence equilibrium $E_4$ at critical parameter value  $L=L^{H}$, subject to the  following conditions:
		\begin{eqnarray}\label{c1}
		\Omega_{1}(L^H)>0, ~\Omega_{3}(L^H)>0, \quad \Omega_{1}(L^H)\Omega_{2}(L^H) =\Omega_{3}(L^H)
		\end{eqnarray}
		and
		\begin{eqnarray}\label{c2}
		\left[\Omega_{1}(L)\Omega_{2}(L)\right]^{\prime}_{L = L^H} \neq \Omega_{3}^{\prime}(L^H).
		\label{eq:}
		\end{eqnarray}
\end{theorem}
\begin{proof}
To detect the presence of a Hopf bifurcation at $L=L_H$ near the coexistence equilibrium point $E_4$, the characteristic equation should take on the following form:
		\begin{equation}\label{characeq}
		(\lambda^{2}(L^{H}) + \Omega_{2}(L^{H}))(\lambda(L^{H})+\Omega_{1}(L^{H}))=0 ,
		\end{equation}
		which has roots $\lambda_{1}(L^{H}) = i \sqrt{\Omega_{2}(L^{H})},$  $\lambda_{2}(L^{H}) = -i \sqrt{\Omega_{2}(L^{H})},$  $\lambda_{3}(L^{H}) = - \Omega_{1}(L^{H})<0,$ then, $\Omega_3(L^{H}) = \Omega_1(L^{H})\Omega_2(L^{H})$.
We shall verify the existence of the transversality condition:		
		\begin{equation}
		\Bigg[\dfrac{\text{d}(Re\lambda_{j}(L))}{\text{d}L}\Bigg]_{L=L^{H}}\not=0,\quad j=1,2.
		\end{equation} 
The following results is obtained by substituting $\lambda_{j}(L) = \alpha(L)+i\beta(L)$ into (\ref{characeq}) and subsequent differentiation with respect to $L$:
		
	\begin{eqnarray}
		Q_{1}(L)\alpha^{\prime}(L)-Q_{2}(L)\beta^{\prime}(L) +Q_{4}(L) &=& 0, \label{u1}\\
		Q_{2}(L)\alpha^{\prime}(L) + Q_{1}(L)\beta^{\prime}(L) + Q_{3}(L)&=&0,\label{u2}
		\end{eqnarray}
		where
		\begin{eqnarray*}
			Q_{1}(L)&=&3\alpha^{2}(L)-3\beta^{2}(L)+\Omega_{2}(L)+2\Omega_{1}(L)\alpha(L),\\
			Q_{2}(L)&=& 6\alpha(L)\beta(L)+2\Omega_{1}(L)\beta(L),\\
			Q_{3}(L)&=&2\alpha(L)\beta(L)\Omega_{1}^{\prime}(L)+\Omega_{2}^{\prime}(L)\beta(L),\\
			Q_{4}(L)&=&\Omega_{2}^{\prime}(L)\alpha(L)+\alpha^{2}(L)\Omega_{1}^{\prime}(L)-\beta^{2}(L)\Omega_{1}^{\prime}(L) + \Omega_{3}^{\prime}(L).
		\end{eqnarray*}
				
\noindent At $L=L^{H},$ $\alpha(L^{H})=0$ and $\beta(L^{H})=\sqrt{\Omega_{2}(L^{H})}$.  We obtain
\begin{eqnarray*}
			Q_{1}(L^{H}) &=& -2 \Omega_{2}(L^{H}),\\
			Q_{2}(L^{H}) &=& 2 \Omega_{1}(L^{H})\sqrt{\Omega_{2}(L^{H})},\\
			Q_{3}(L^{H}) &=& \Omega_{2}^{\prime}(L^{H})\sqrt{\Omega_{2}(L^{H})},\\
			Q_{4}(L^{H}) &=& \Omega_{3}^{\prime}(L^{H}) - \Omega_{2}(L^{H})\Omega_{1}^{\prime}(L^{H}).
		\end{eqnarray*}	
The following expression is derived by solving equations \eqref{u1} and \eqref{u2} for $\alpha^{\prime}(L^{H})$:
		\begin{eqnarray*}
			\left[\frac{\text{d}Re(\lambda_{j}(L))}{\text{d}L}\right]_{L=L^{H}}&=&\alpha^{\prime}(L^{H})\\&=&-\frac{Q_{4}(L^{H})Q_{1}(L^{H})+Q_{3}(L^{H})Q_{2}(L^{H})}{Q_{1}^{2}(L^{H})+Q_{2}^{2}(L^{H})}\\
			&=&\frac{\Omega_{3}^{\prime}(L^{H})-\Omega_{2}(L^{H})\Omega_{1}^{\prime}(L^{H})-\Omega_{1}(L^{H})\Omega_{2}^{\prime}(L^{H})}{2\left(\Omega_{2}(L^{H})+\Omega_{1}^{2}(L^{H})\right)}\not=0
		\end{eqnarray*}
				
\noindent on condition that
\begin{align*}
\left[\Omega_{1}(L)\Omega_{2}(L)\right]^{\prime}_{L = L^{H}} \neq \Omega_{3}^{\prime}(L^{H}).
\end{align*}

\noindent Therefore, the model \eqref{Mainsystem} undergoes a Hopf bifurcation near the coexistence equilibrium point $E_4$ when $L=L^H$.
\end{proof}

\subsubsection{Saddle-node bifurcation}
A saddle-node bifurcation happens when two equilibrium points—one stable and one unstable—come together and cancel each other out as a parameter passes a certain critical value. When this occurs, the stability of these points changes suddenly, leading to new patterns of behavior in the system. 

\begin{theorem}\label{thm:saddle-node}
The model \eqref{Mainsystem} exhibits  a saddle-node bifurcation near the point $E_4$ at the critical parameter value $L=L^{SN}$, provided $\Omega_3=0$.
\end{theorem}

\begin{proof}
The Jacobian matrix $\mathbb{J}_{E_4}$ has a zero eigenvalue when $\det(\mathbb{J}_{E_4}) = 0$, corresponding to the critical point $L^{SN}$. Let $U$ and $V$ be the eigenvectors associated with this zero eigenvalue, corresponding to the matrices $\mathbb{J}_{E_4}$ and $\mathbb{J}^T_{E_4}$, respectively. We then obtain:
\[
U = (\bar{u}_1, \bar{u}_2, \bar{u}_3)^T
\]
\[
V = (\bar{v}_1, \bar{v}_2, \bar{v}_3)^T,
\]
where the components are given by:
\[
\bar{u}_1 = \delta_1, \quad \bar{u}_2 = -\frac{F_{31}}{F_{32}}\delta_1, \quad \bar{u}_3 = \frac{F_{12}F_{31} - F_{11}F_{32}}{F_{13}F_{32}}\delta_1,
\]
\[
\bar{v}_1 = \delta_2, \quad \bar{v}_2 = -\frac{F_{13}}{F_{23}}\delta_2, \quad \bar{v}_3 = \frac{F_{21}F_{13} - F_{11}F_{23}}{F_{23}F_{31}}\delta_2.
\]

\noindent Here, $\delta_1$ and $\delta_2$ are arbitrary constants, and $F_{ij}$ for $i, j = 1, 2, 3$ are elements of the $\Omega_3$.

\noindent Let $W = (W_1, W_2, W_3)^T$, where $W_1, W_2$, and $W_3$ are defined in the model \eqref{Mainsystem}. We then compute:
\[
V^T W_{L}(E_4, L^{SN}) = (\bar{v}_1, \bar{v}_2, \bar{v}_3) \left( -a_0 S \left( 1 - \frac{S+I}{K} \right), 0, 0 \right)^T
\]
\[
= -a_0 \delta_2 S \left( 1 - \frac{S+I}{K} \right) \neq 0,
\]
and 
\[
V^T [D^2 W(E_4, L^{SN})(U, U)] \neq 0.
\]
\noindent Therefore, the model \eqref{Mainsystem} undergoes a saddle-node bifurcation at the point $E_4$ when the parameter $L$ reaches the critical value $L^{SN}$.
\end{proof}

\begin{remark}\label{remark:SNHTC}
Numerical analysis obtained with the help of MATCONT package \cite{G05} in MATLAB R2024b, as presented in Figure \ref{fig:one_param_birfucation}, reveals a range of one-parameter bifurcations in the system as the Allee threshold, $L$, is varied from $-1$ to $1$. We identify saddle-node, Hopf, and transcritical bifurcations in Figure \ref{fig:one_param_birfucation}(a). The saddle-node bifurcation is located at $L^{SN}= 0.2396$, with the equilibrium $(S, I, P) = (1.8642, 0.9760, 1.8254)$. The Hopf bifurcation occurs at $L^{H} = 0.2184$, with the equilibrium $(S, I, P) = (1.6746, 1.0295, 1.5817)$ and a first Lyapunov coefficient of $7.2293$. The transcritical bifurcation is observed in the weak Allee effect regime at $L^{TC}= -0.4312$, with the equilibrium $(S, I, P) = (0.4444, 1.5, 0)$. Figure \ref{fig:one_param_birfucation}(b) illustrates a transcritical bifurcation observed when varying the susceptible prey aggregation constant, $r$. This bifurcation occurs at $r^{TC} = 0.7641$, with the corresponding equilibrium point at $(S, I, P) = (3.6099, 0, 4.0698)$.
\end{remark}

\begin{figure}[hbt!]
\begin{center}
\subfigure[]{
    \includegraphics[width=6.15cm, height=6.15cm]{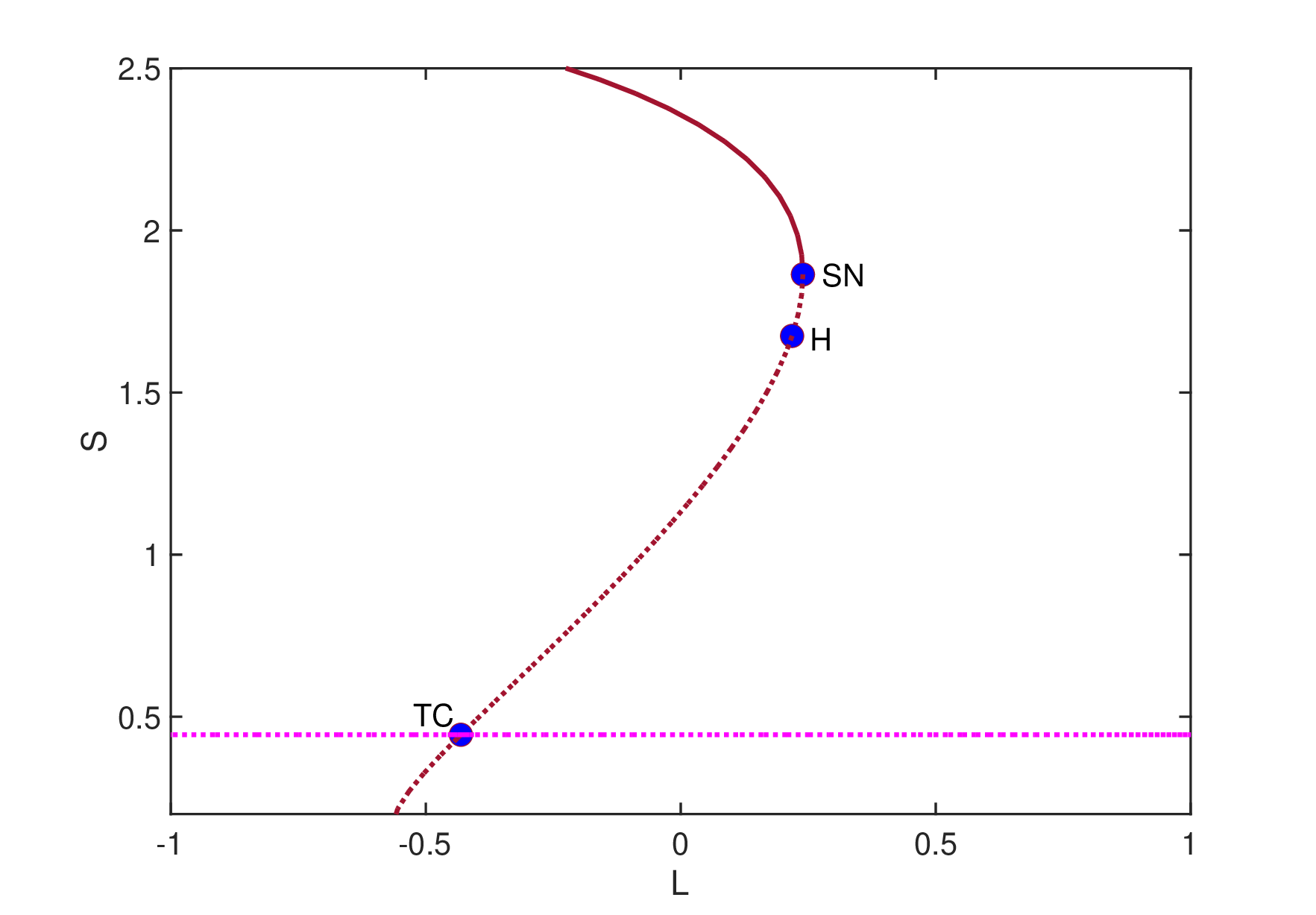}}
\subfigure[]{    
    \includegraphics[width=6.15cm, height=6.15cm]{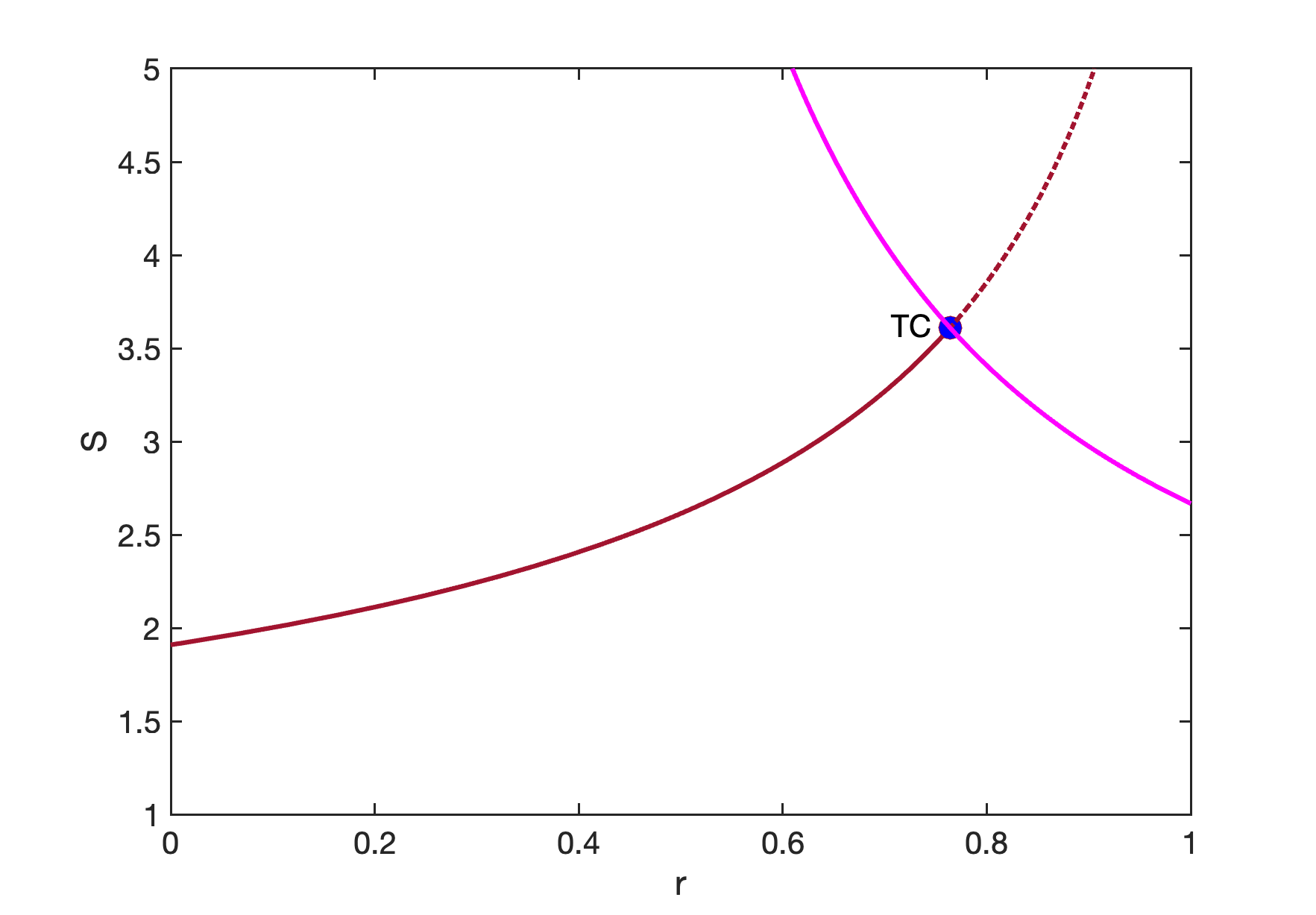}}
\end{center}
\caption{One parameter bifurcation plots illustrating the effects of varying the Allee threshold, and also the aggregation constant. (a) $L$ as a bifurcation parameter  (b) $r$ as a bifurcation parameter. All other parameters are fixed and given as $a_0=3,~a_1=0.4,~a_2=0.8,~r=0.5,~d_0=0.4,~d_1=0.7,~d_2=0.3,~d_3=0.4,~e_0=0.9,~K=4,~L=-0.5$. For clarity, the following notation is used in the figures: SN denotes a saddle-node bifurcation, TC denotes a transcritical bifurcation, and H denotes a Hopf bifurcation. Stable equilibria are represented by solid lines, and unstable equilibria are represented by dashed lines}
\label{fig:one_param_birfucation}
\end{figure}

\subsection{Co-dimension two bifurcation}
In this subsection, we examine the possibility of obtaining co-dimension two bifurcations in model \eqref{Mainsystem}, a feature that allows for more intricate dynamics than simpler models due to the simultaneous variation of two key parameters.

\subsubsection{Cusp bifurcation}
A cusp (CP) bifurcation is a co-dimension two local bifurcation where two saddle node points collide. 

Using the MATCONT package \cite{G05} in MATLAB R2024b  to perform a continuation of the saddle node bifurcation (i.e., $L = 0.2396$ at $(1.8642,0.9760,1.8254,1.5817)$) with $(L, e_0)$ as free parameters, we identify a cusp point ($CP_1$) at $L = 2.5747$, occurring at $(2.9654,0.0.7085,0)$ with $e_0 = 0.1349$. The corresponding eigenvalues are numerically computed as $\lambda_1 = 0$ and $\lambda_{2,3} = -0.0718 \pm 0.3407$. This two-parameter bifurcation is illustrated in Figure \ref{fig:two_param_birfucation}(b). We also identified $CP_2$ at  $L = 4.0147$, occurring at $(3.6883,0.5596,0)$ with $e_0 = 0.1085$. The corresponding eigenvalues are numerically computed as $\lambda_1 = 0,~\lambda_2=0.3144,$ and $\lambda_{3} = -0.0971$.

\subsubsection{Zero-Hopf Equilibrium}
In a 3-dimensional system of differential equations, a zero-Hopf (ZH) equilibrium occurs when the system has a special combination of eigenvalues: one eigenvalue is zero, and another pair of eigenvalues are purely imaginary.

Using the MATCONT package \cite{G05} in MATLAB R2024b  to perform a continuation of the Hopf bifurcation (i.e., $L = 0.2184$ at $(1.6746,1.0294,1.5817)$) with $(L, a_0)$ as free parameters, we identify a zero-Hopf point ($ZH$) at $L = -1.6111$, occurring at $(0.4444,1.5,0)$ with $a_0 = 1.2780$. The corresponding eigenvalues are numerically computed as $\lambda_1 = 0$ and $\lambda_{2,3} = \pm 0.9665i$. This two-parameter bifurcation is illustrated in Figure \ref{fig:two_param_birfucation}(a). Note that $ZH$ in Figure \ref{fig:two_param_birfucation}(b) is not biologically feasible since the predator population is negative (i.e., $P=-0.2104$), we will not consider it further.

\subsubsection{Generalized Hopf  bifurcation}
The generalized Hopf (GH) (or Bautin) bifurcation, a co-dimension two local bifurcation, is characterized by a zero first Lyapunov coefficient at a coexistence equilibrium with purely imaginary eigenvalues. This bifurcation serves as a boundary between subcritical and supercritical Hopf bifurcations in parameter space.

Using the MATCONT package \cite{G05} in MATLAB R2024b  to perform a continuation of the Hopf bifurcation (i.e., $L = 0.2184$ at $(1.6746,1.0294,1.5817)$) with $(L, a_0)$ as free parameters, we identify a generalized-Hopf point ($GH$) at $L = -1.6507$, occurring at $(0.5016,1.4688,0.0735)$ with $a_0 = 1.2485$. The corresponding eigenvalues are numerically computed as $\lambda_1 = 0.0015$ and $\lambda_{2,3} = \pm 1.0377i$. This two-parameter bifurcation is illustrated in Figure \ref{fig:two_param_birfucation}(a). The second Lyapunov coefficient is $-1.2394$.

\subsubsection{Bogdanov-Takens bifurcation}
The Bogdanov-Takens (BT) bifurcation describes a significant change in a system's behavior, triggered by the simultaneous variation of two parameters, at a point where the system's stability analysis reveals two zero eigenvalues.

Using the MATCONT package \cite{G05} in MATLAB R2024b  to perform a continuation of the saddle node bifurcation (i.e., $L = 0.2396$ at $(1.8642,0.9760,1.8254,1.5817)$) with $(L, e_0)$ as free parameters, we identify a Bogdanov-Takens point ($BT_2$) at $L = 4.4253$, occurring at $(3.8983,0.5192,0.3777)$ with $e_0 = 0.1704$. The corresponding eigenvalues are numerically computed as $\lambda_1 = 0.3588,~\lambda_2=0,$ and $\lambda_{3} = 0$. This two-parameter bifurcation is illustrated in Figure \ref{fig:two_param_birfucation}(b). Note that $BT_1$ is not biologically feasible since the predator population is negative (i.e., $P=-0.1683$), we will not consider it further.

\begin{figure}[hbt!]
\begin{center}
\subfigure[]{
    \includegraphics[width=6.15cm, height=6cm]{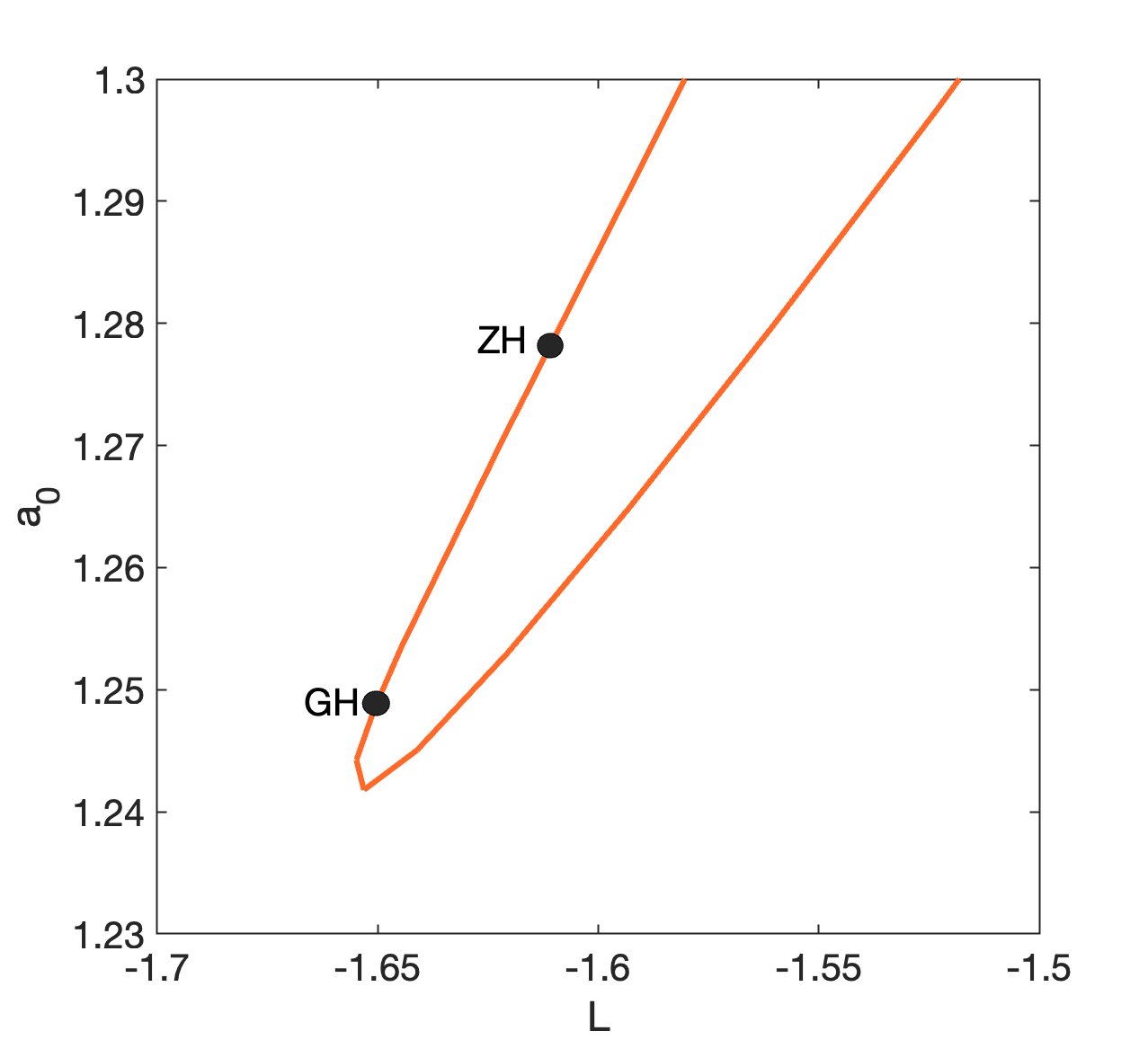}}
\subfigure[]{    
    \includegraphics[width=6.15cm, height=6cm]{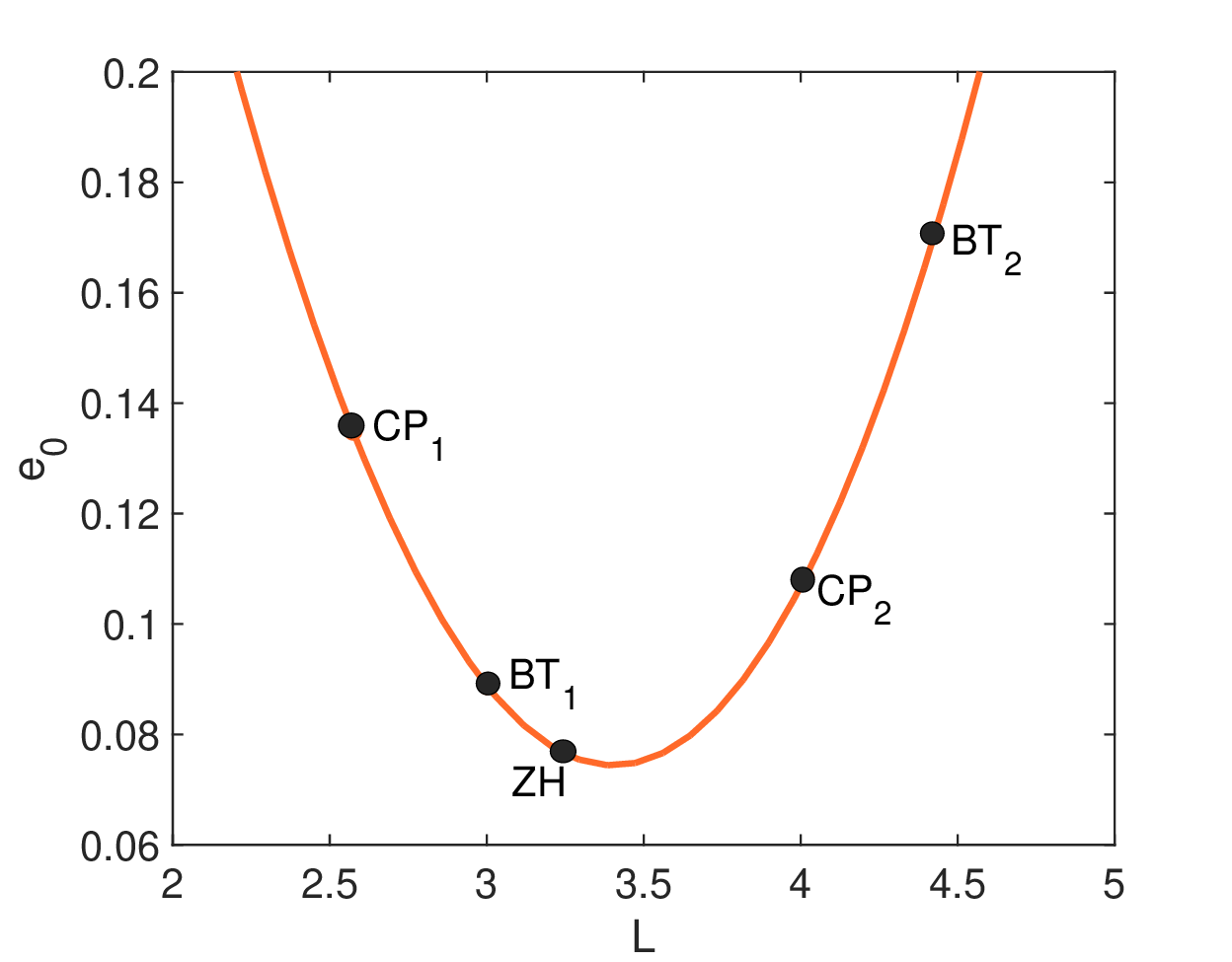}}
\end{center}
\caption{Two parameter bifurcation plots.  (a) $L-a_0$ parametric space depicting $GH$ and $ZH$ in the weak Allee regime  (b) $L-e_0$ parametric space depicting $CP_1,~CP_2,~BT_1,~BT_2,$ and $ZH$ in the strong Allee regime. All other parameters are fixed and given as $a_0=3,~a_1=0.4,~a_2=0.8,~r=0.5,~d_0=0.4,~d_1=0.7,~d_2=0.3,~d_3=0.4,~e_0=0.9,~K=4,~L=-0.5$}
\label{fig:two_param_birfucation}
\end{figure}

\subsection{Regions of parametric space: $L$ vs $r$}
In this subsection, we investigate how variations in the Allee threshold parameter ($L$) and the aggregation constant ($r$) influence the model's long-term behavior, leading to either a stable coexistence equilibrium or extinction equilibria. For the stable coexistence region (Region I), the infectious prey does not drive the susceptible prey to extinction, and the predator population is able to sustain itself. This occurs under conditions where the Allee effect is not too strong, and susceptible prey aggregation provides enough protection for a stable population structure. Additionally, the extinction region are classified into two distinct regions:
\begin{itemize}
\item[(i)] an infected prey-free region (Region II), where the infected prey ($I$) is eradicated while the susceptible prey ($S$) and predators ($P$) persist. Biologically, this could mean that infection cannot be sustained due to unfavorable conditions for disease transmission (e.g., low host density, high prey aggregation reducing contacts). This is an important ecological outcome because it suggests that under certain conditions, disease eradication is possible without collapsing the entire prey-predator system.
\item[(ii)] finite time extinction of the susceptible prey population region (Region III), ultimately leading to the extinction of both the infected prey and the predator population (or total collapse) . This happens when the Allee effect is strong, meaning that at low susceptible prey densities, the reproduction rate is too low to sustain the population. Once susceptible prey vanish, the infected prey (which depend on them for transmission) and the predators (which rely on them for food) also die out.
\end{itemize}
These distinct regions are illustrated in Figure \ref{fig:Allee vs aggregation}, which delineates the parameter regions associated with stable coexistence, disease eradication and total collapse.

\begin{figure}[hbt!]
\begin{center}
    \includegraphics[width=11cm, height=9cm]{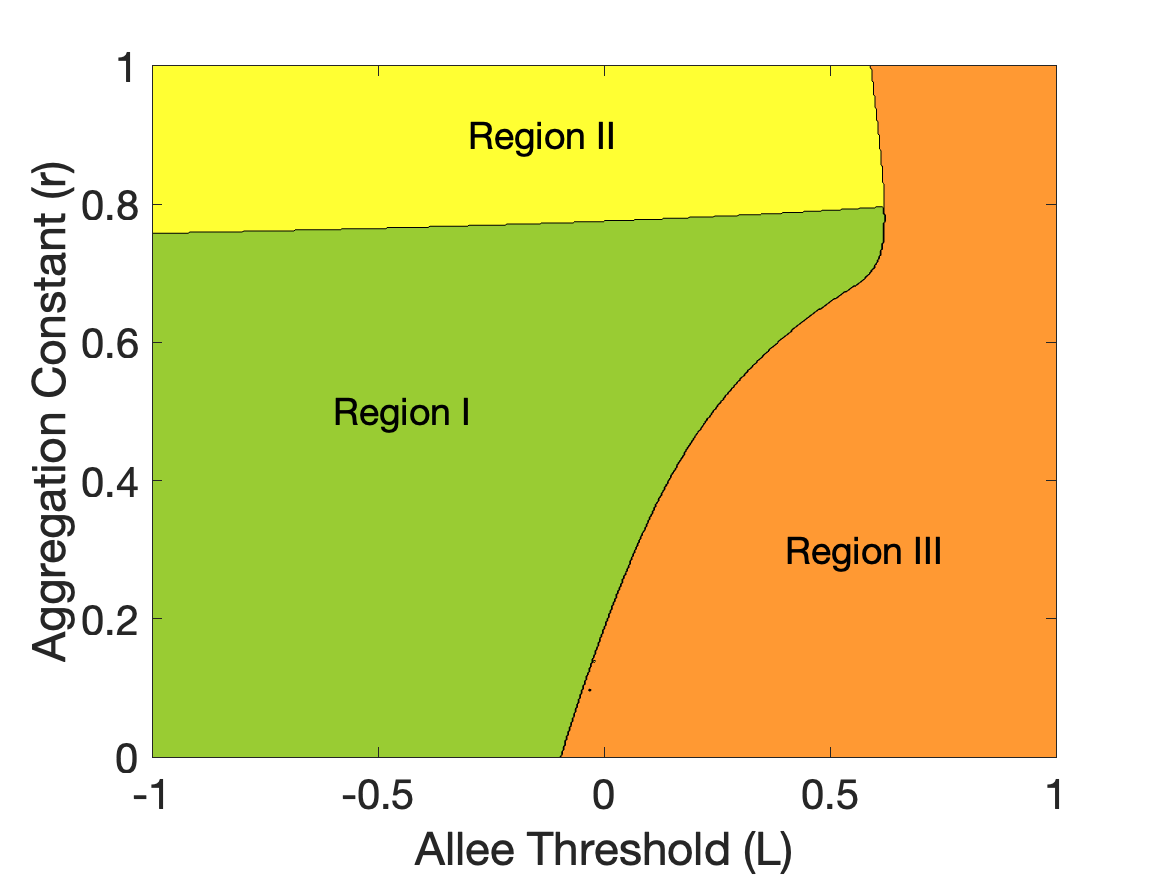}
\end{center}
\caption{Bifurcation diagram illustrating the effects of varying $L$ and $r$ on the population dynamics in model \eqref{Mainsystem}.  Stable coexistence region (green), infectious prey free region (yellow), and finite time extinction of susceptible prey region or total collapse (orange). The parameters are fixed as  $a_0=3,~a_1=0.4,~a_2=0.8,~d_0=0.4,~d_1=0.7,~d_2=0.3,~d_3=0.4,~e_0=0.9,~K=4$ with initial condition  $(S(0),I(0),P(0))=(2,1,3)$}
\label{fig:Allee vs aggregation}
\end{figure}

\section{Disease management by targeting susceptible prey aggregate}\label{sec: Disease managemment}
\noindent Our model reveals a critical threshold for disease persistence: by manipulating the aggregated constant of the susceptible prey under a weak Allee effect, we demonstrate the potential for complete elimination of infected prey, see Figure \ref{fig:timeseries:diseaseControl}.  From an ecological perspective, this suggests that a sufficiently low density of susceptible prey population can interrupt the chain of infection, preventing the disease from maintaining a foothold in the population. It underscores the importance of susceptible prey population availability in driving disease dynamics.

\begin{theorem}\label{thm:inf-prey extinction}
For the SIP model described by the equations in model \eqref{Mainsystem}, under a specific parameter set and initial data $(S(0), I(0), P(0))$ that converges uniformly to a stable coexistence state, there exists a critical threshold ($r^*$), i.e., $0<r<r^*<1$ such that the solution for the infectious prey population starting from the same initial data will eventually go extinct, i.e. $\displaystyle\lim_{t \to \infty} I(t) = 0$.
\end{theorem}

\begin{proof}
To establish the extinction of the infectious prey population $I(t)$ under the model \eqref{Mainsystem}, we analyze the behavior of the differential equation governing $I$:

\begin{equation*}
\frac{dI}{dt} = -a_1 I + e_0 S I - d_1 I P.
\end{equation*}

\noindent Factoring out $I$, we obtain:

\begin{equation*}
\frac{dI}{dt} = (-a_1 + e_0 S - d_1 P) I.
\end{equation*}

\noindent For the population of infected prey to asymptotically approach extinction, it is necessary that

\begin{equation*}
\frac{dI}{dt} < 0 \quad \text{for sufficiently large } t.
\end{equation*}

\noindent Since $I(t) \geq 0$ for all $t$, this condition is satisfied if and only if:

\begin{equation} \label{eq:extinction_cond}
-a_1 + e_0 S - d_1 P < 0 \quad \Rightarrow \quad a_1 > e_0 S - d_1 P.
\end{equation}

\noindent The infectious prey-free equilibrium $E_3 = (S_3, 0, P_3)$ satisfies the model's equilibrium conditions, thus if the model stabilizes at $E_3,$ the infected prey population will disappear over time. Substituting $S = S_3$ and $P = P_3$ into \eqref{eq:extinction_cond}, the extinction criterion simplifies to:

\begin{equation} \label{eq:stability_condition}
a_1 > e_0 S_3 - d_1 P_3.
\end{equation}

\noindent Thus, we obtain the key inequality:

\begin{equation}
a_1 > e_0 \left(\frac{a_2}{d_2} \right)^{1/r} + \frac{d_1 a_0 \left(\left(\frac{a_2}{d_2} \right)^{1/r} \right)^{1-r} \left(\left(\frac{a_2}{d_2} \right)^{1/r} - K \right) \left(\left(\frac{a_2}{d_2} \right)^{1/r} - L \right)}{d_0 K}.
\end{equation}

\noindent The inequality above implicitly defines a critical threshold $r^*$, such that for all $0<r<r^*<1$, the inequality is satisfied, leading to the extinction of the infectious prey. This establishes that there exists a bifurcation threshold of the model that determines whether $I(t)$ persists or vanishes as $t \to \infty$.

Thus, under the given assumptions, we conclude that the infectious prey population $I(t)$ eventually go extinct.
\end{proof}

\begin{figure}[hbt!]
\begin{center}
\subfigure[]{
    \includegraphics[width=6.15cm, height=6.15cm]{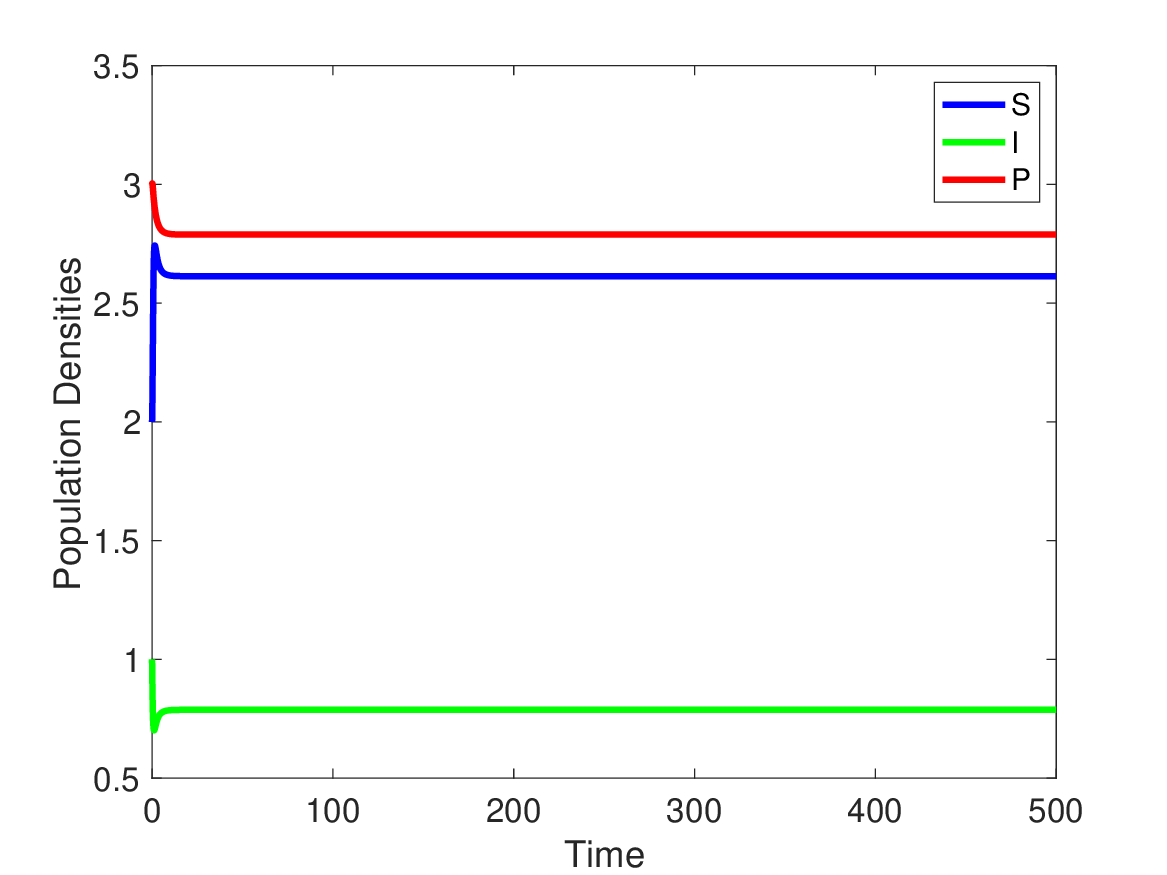}}
\subfigure[]{    
    \includegraphics[width=6.15cm, height=6.15cm]{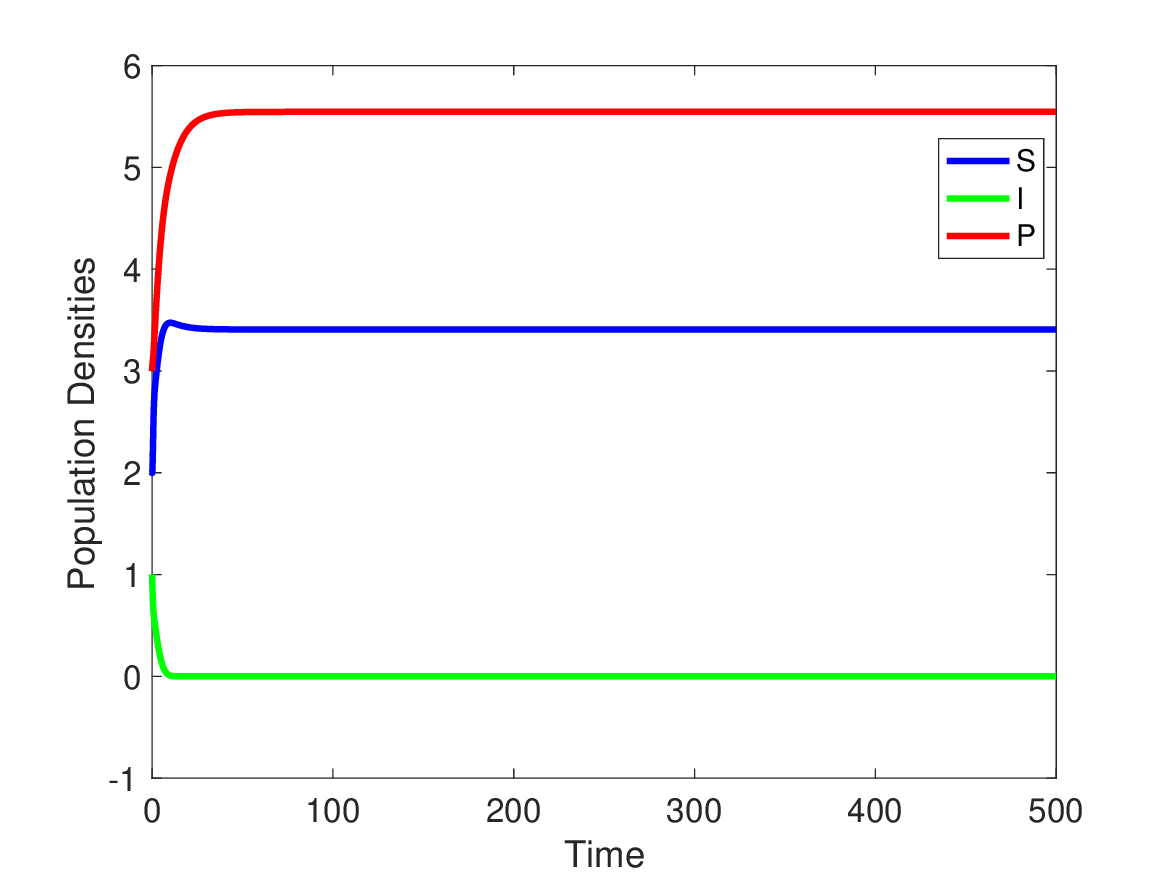}}
\subfigure[]{
    \includegraphics[width=6.15cm, height=6.15cm]{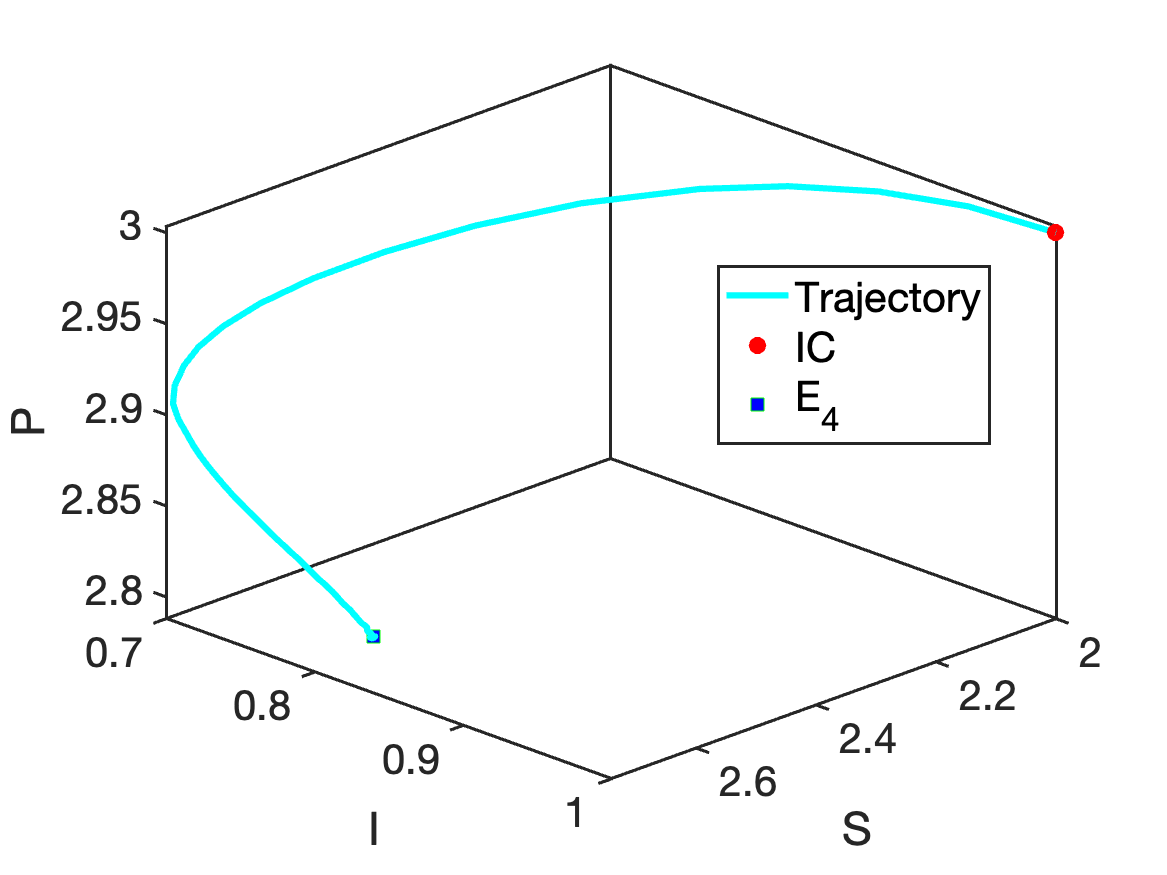}}
\subfigure[]{    
    \includegraphics[width=6.15cm, height=6.15cm]{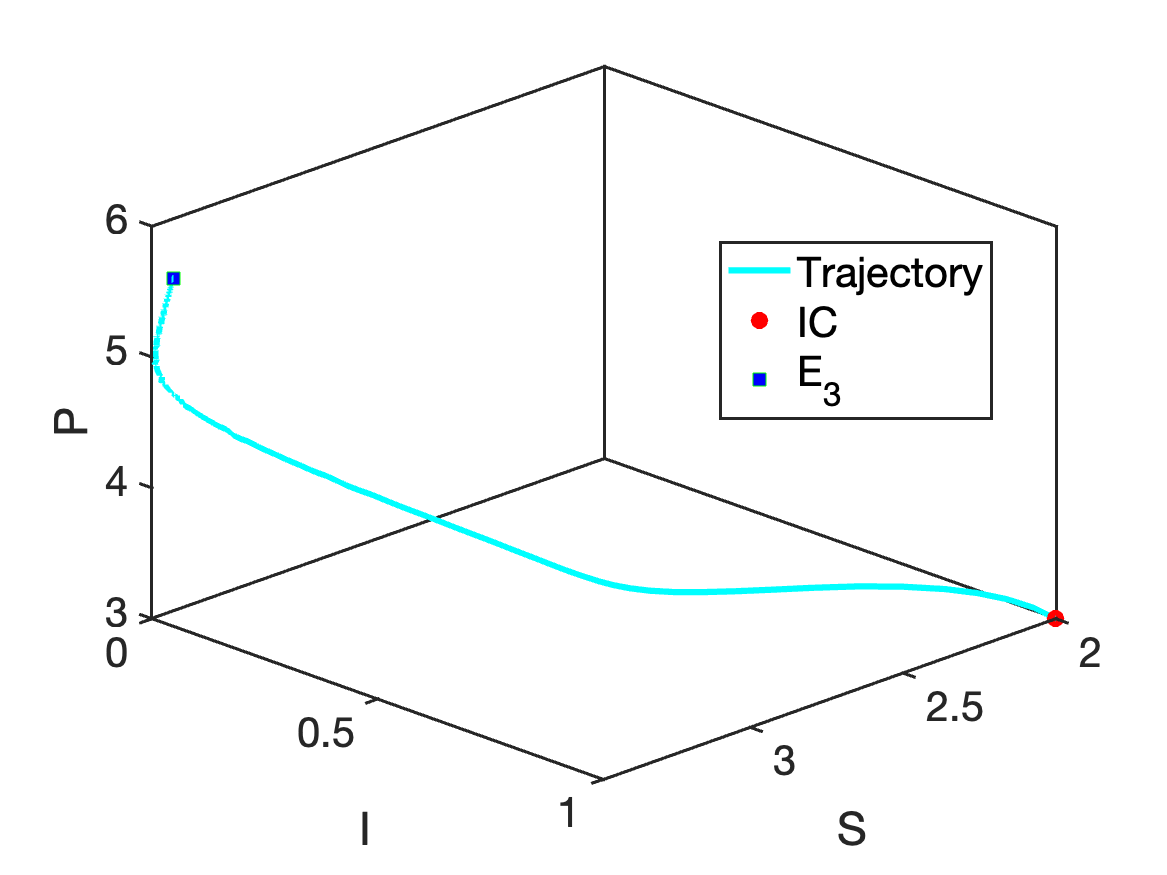}}    
\end{center}
\caption{Time series and phase portrait plots illustrating the transition from a stable coexistence state to an infected prey free state as the aggregation constant is altered. (a) stable coexistence equilibrium point $E_4=(2.61341,0.787546,2.78867)$, here $r=0.5$ (b) extinction of the infected population $E_3=(3.4077,0,5.54573)$, here $r=0.8$. (c) 3D phase portrait of the SIP model for $r=0.5$ (d)  3D phase portrait of the SIP model for $r=0.8$.  All other parameters are fixed and given as $a_0=3,~a_1=0.4,~a_2=0.8,~d_0=0.4,~d_1=0.7,~d_2=0.3,~d_3=0.4,~e_0=0.9,~K=4,~L=-0.5$. with initial condition $(S(0),I(0),P(0))=(2,1,3)$}
\label{fig:timeseries:diseaseControl}
\end{figure}

\section{Allee Effect-Driven Finite Time Extinction}\label{sec: Allee-driven extinction}
\noindent This section explores the novel concept of Allee effect-driven finite time extinction in susceptible prey populations. We move beyond traditional asymptotic extinction analyses to investigate the potential for rapid population collapse within a defined time frame. By focusing on the Allee effect's influence, we aim to identify the specific conditions under which susceptible prey can abruptly disappear. We will employ mathematical modeling and analytical techniques to provide a rigorous framework for understanding this phenomenon, offering insights into the transient dynamics and potential for rapid ecosystem shifts.

\begin{theorem}\label{thm:finite_time_extinction}
For the SIP model described by the equations in model \eqref{Mainsystem}, consider a specific parameter set and initial data $(S(0), I(0), P(0))$ that lead to uniform convergence to a stable coexistence state. If $S(0)<L^*$, where $L^*$ is the critical Allee threshold, then the susceptible prey population will eventually go extinct in finite time, starting from the same initial conditions.
\end{theorem}

\begin{proof}
Consider the first equation in the model \eqref{Mainsystem} for the susceptible prey population \( S(t) \):
\begin{align*}
\frac{dS}{dt} = a_0 S \left( 1 - \frac{S + I}{K} \right) (S - L) - d_0 S^r P - e_0 S I.
\end{align*}

\noindent To establish finite-time extinction, we first simplify the equation by neglecting the terms involving \( P \) and \( I \), leading to the following inequality:
\begin{align*}
\frac{dS}{dt} \leq a_0 S \left( 1 - \frac{S}{K} \right) (S - L).
\end{align*}

\noindent The important factor here is the term \( (S - L) \), which causes \( \frac{dS}{dt} \) to be negative when \( S < L \). We now assume that \( S(0) < L^* \), where \( L^* \) is a threshold such that \( S(0) < L \). Under this assumption, \( (S - L) < 0 \), implying that \( \frac{dS}{dt} < 0 \).

\noindent Moreover, as long as \( S(t) < L \), the term \( (S - L) \) remains negative, ensuring that \( S(t) \) continues to decrease. Thus, we conclude that if \( S(0) < L^* \), then:
\begin{align*}
\lim_{t \to t^*} S(t) = 0,
\end{align*}
and the susceptible prey population \( S(t) \) goes extinct in finite time.
\end{proof}

\begin{remark}\label{remark:FTE}
Figure \ref{fig:timeseriesFTE} presents a novel sequence of dynamic transitions in the SIP model \eqref{Mainsystem}. We see a shift from stable coexistence to an unstable predator-free equilibrium, and finally to the finite time extinction of susceptible prey, as the Allee threshold ($L$) progresses from weak (i.e., $L=-0.8$  and $L=0$) to strong (i.e., $L=0.1$). Ecologically, this process also results in the subsequent extinction of infected prey and predators, thus a total collapse of the population. This observation highlights a unique and complex response to increasing Allee effect threshold in population dynamics.
\end{remark}

\begin{figure}[hbt!]
\begin{center}
\subfigure[]{
    \includegraphics[width=6.15cm, height=6.15cm]{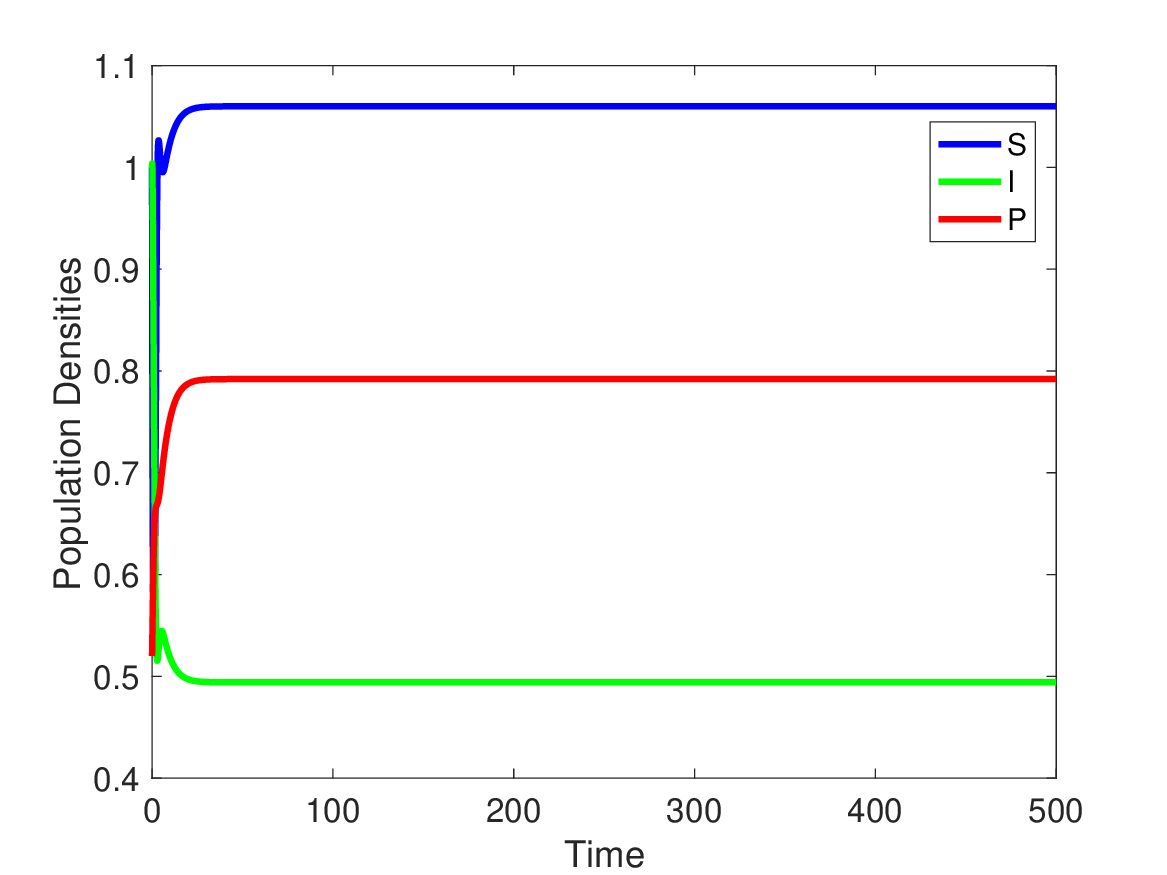}}
\subfigure[]{    
    \includegraphics[width=6.15cm, height=6.15cm]{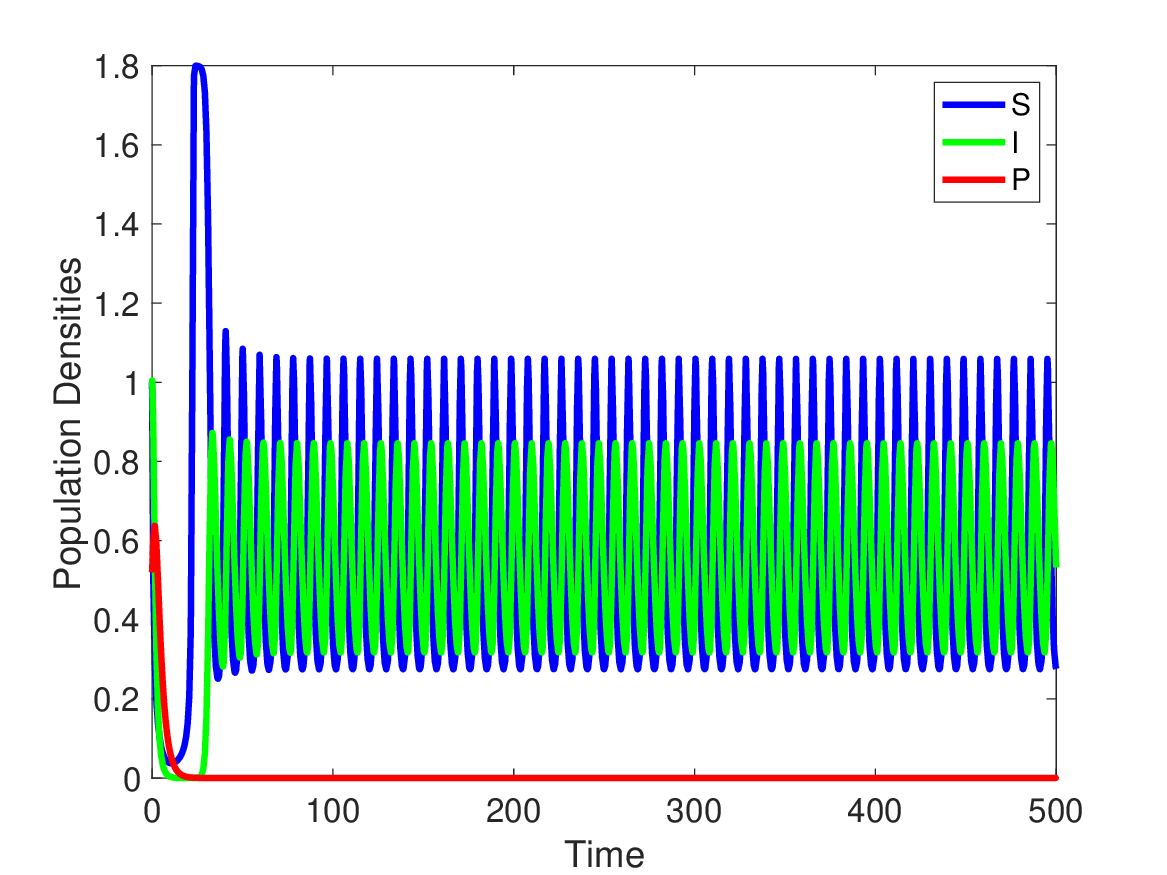}}
\subfigure[]{    
    \includegraphics[width=6.15cm, height=6.15cm]{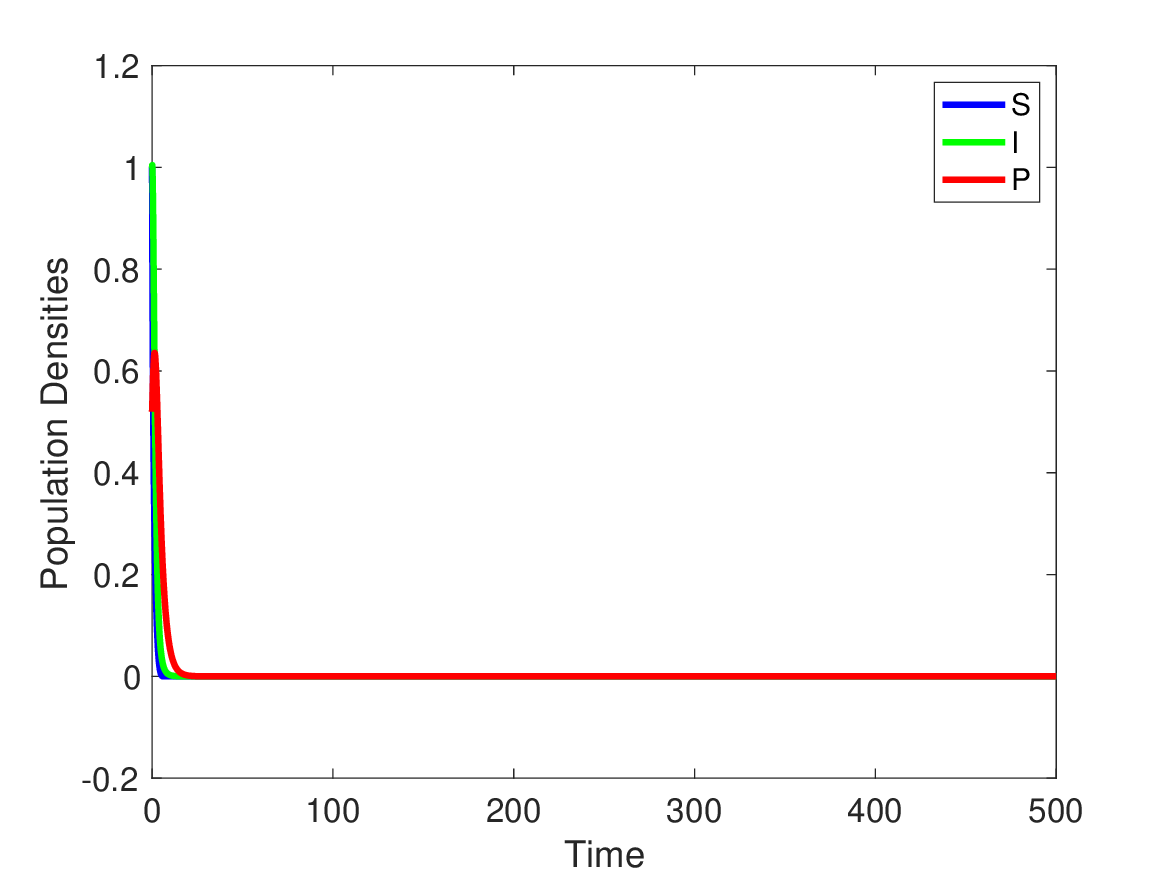}}
\end{center}
\caption{Time series plots illustrating the shift from a weak to a strong Allee effect, highlighting its impact on population dynamics. (a) stable interior equilibrium point ($1.06,0.494,0.792$), here $L=-0.8$ (b) unstable dynamics when $L=0$, (c) total population collapse, beginning with the finite-time extinction of the susceptible prey occurring at $L=0.1$ and $t=6.2$, followed by the eventual extinction of both the infected prey and predator population. All other parameters are fixed and given as $a_0=3,~a_1=0.5,~a_2=0.35,~d_0=0.4,~d_1=0.6,~d_2=0.1,~d_3=0.5,~e_0=0.92,~K=1.8,~r=0.5$ with initial condition $(S(0),I(0),P(0))=(1,1,0.52)$}
\label{fig:timeseriesFTE}
\end{figure}

\section{Conclusion}\label{sec:conclusion}
\noindent This study introduces a novel mathematical framework for analyzing the interplay between disease dynamics, predator-prey interactions, and Allee effects in ecological systems \cite{kumar2023dynamics,kang2014dynamics}. By incorporating both weak and strong Allee thresholds alongside prey aggregation behavior, it demonstrates that disease outbreaks and population stability are highly sensitive to specific ecological parameters \cite{rana2022complex,firdiansyah2020dynamics}. The findings reveal conditions under which finite-time extinction of infected prey can occur (please see Theorem \ref{thm:finite_time_extinction} for details), offering valuable insights into disease control strategies in wildlife populations. This result aligns with ecological findings \cite{ma2023asymptotic,kang2014dynamics}. Furthermore, through stability and bifurcation analyses, the existence of multiple equilibrium states and transitions, including transcritical, Hopf, and saddle-node bifurcations, which govern the long-term behavior of the system, are established, see Figures \ref{fig:equilibriumPoints} and \ref{fig:one_param_birfucation}.

A particularly significant result of this study is the identification of parameter regimes where disease-driven extinction can be mitigated through prey aggregation strategies, leading to an extinction threshold for the infected population, see Theorem \ref{thm:inf-prey extinction} and Figure \ref{fig:timeseries:diseaseControl}. This insight underscores the critical role of social and behavioral mechanisms in shaping disease persistence in ecological systems under weak Allee effect. Moreover, the observed Allee effect-driven finite-time extinction highlights an important phenomenon in population ecology: small populations are not only vulnerable to environmental fluctuations but may also collapse rapidly due to intrinsic growth constraints. The mathematical analysis provides a robust theoretical foundation for understanding such tipping points in ecological networks.

The implications of our results extend beyond theoretical ecology. In conservation biology, understanding how aggregation and Allee effect mechanisms influence disease spread can inform intervention strategies aimed at preserving at-risk species. We presented a bifurcation diagram in Figure \ref{fig:Allee vs aggregation} to depict the long term dynamics of the model when the aggregation constant and Allee threshold are modulated simultaneously. Additionally, the co-dimension two bifurcation structures uncovered in our model suggest potential applications in epidemiology, particularly in predicting and controlling pathogen-driven collapses in structured populations, see Figure \ref{fig:two_param_birfucation}. Figure \ref{fig:two_param_birfucation}(a) reveals that both the zero-Hopf and generalized Hopf points occur within the weak Allee effect regime, whereas Figure \ref{fig:two_param_birfucation}(b) shows that the cusp and Bogdanov-Takens points arise in the strong Allee effect regime. Thus, these bifurcations highlight the sensitivity of eco-epidemiological systems to population thresholds, emphasizing the need for conservation strategies that mitigate strong Allee effects to prevent irreversible population declines. The methodologies employed in this study, including analytical bifurcation theory and numerical continuation methods, can be extended to explore broader classes of eco-epidemiological models incorporating stochasticity, spatial heterogeneity, and multi-species interactions \cite{lunn2021spatial}.

Future research should aim to refine these models by incorporating empirical data to validate theoretical predictions and assess the real-world applicability of our findings. Additionally, the exploration of higher-dimensional extensions, such as models including multiple predator species or seasonally varying parameters, could yield deeper insights into the resilience of ecological communities \cite{panigoro2019dynamics,singh2024population}. By integrating machine learning techniques and data assimilation methods, future studies can enhance model predictability and improve decision-making in conservation and disease management \cite{ponnambalam2016multi}.

Furthermore, investigating the role of evolutionary pressures, such as genetic adaptation in prey and predator populations, could offer a more comprehensive understanding of long-term ecological stability \cite{shaw2023rapid,saakian2022gene}. Another promising direction involves the incorporation of spatially explicit models, allowing researchers to examine how localized interactions and migration patterns influence disease dynamics and population persistence \cite{lunn2021spatial}. The use of agent-based modeling techniques could further enhance our understanding of individual-level behaviors and their emergent effects at the population level \cite{ponnambalam2016multi}.

\section*{Acknowledgment}
\noindent KAF, ZO, and DL would like to express their profound gratitude to CURM mini-grant and NSF grant DMS-1722563.

\section*{Declaration of competing interest}
\noindent The authors declare that they have no known competing financial interests or personal relationships that could have appeared to influence the work reported in this paper.

\section*{Availability of data and materials}
\noindent Not applicable.
%


\bibliographystyle{unsrt}
\bibliography{manu1.bib}

\end{document}